\documentclass[notitlepage,11pt,reqno]{amsart}
\usepackage{amssymb,amsmath,amscd,amsthm,tocvsec2,soul}
\usepackage[normalem]{ulem}
\usepackage[mathscr]{eucal}
\usepackage[breaklinks=true]{hyperref}
\usepackage[letterpaper]{geometry}
\usepackage{xcolor}

\newcommand{\pH}{\operatorname{^pH}}
\newcommand{\wt}{\operatorname{wt}}
\newcommand{\conv}{\operatorname{conv}}
\newcommand{\Mod}{\operatorname{-mod}}
\newcommand{\Ind}{\operatorname{Ind}}
\newcommand{\Res}{\operatorname{Res}}
\newcommand{\supp}{\operatorname{supp}}
\newcommand{\ch}{\operatorname{ch}}

\newcommand{\Hom}{\operatorname{Hom}}
\newcommand{\Ext}{\operatorname{Ext}}
\newcommand{\rk}{\operatorname{rk}}

\newcommand{\g}{\mathfrak{g}}
\newcommand{\n}{\mathfrak{n}}
\newcommand{\h}{\mathfrak{h}}
\newcommand{\fl}{\mathfrak{l}}
\newcommand{\fs}{\mathfrak{s}}
\newcommand{\halpha}{\check{\alpha}}
\newcommand{\hpi}{\check{\pi}}
\newcommand{\R}{\ensuremath{\mathbb{R}}}
\newcommand{\C}{\ensuremath{\mathbb{C}}}
\newcommand{\Z}{\ensuremath{\mathbb{Z}}}
\newcommand{\lie}[1]{\ensuremath{\mathfrak{#1}}}
\newcommand{\OO}{\mathscr{O}}
\newcommand{\IC}{\mathscr{IC}}
\newcommand{\tu}{\tilde{U}}
\newcommand{\cV}{\mathscr{V}}
\newcommand{\cU}{\mathscr{U}}
\newcommand{\cP}{\mathscr{P}}

\newtheorem{theo}[equation]{Theorem}
\newtheorem{lem}[equation]{Lemma}
\newtheorem{cor}[equation]{Corollary}
\newtheorem{pro}[equation]{Proposition}

\theoremstyle{definition}

\newtheorem{re}[equation]{Remark}

\numberwithin{equation}{section}
\numberwithin{figure}{section}

\begin{document}
\title{The weights of simple modules in Category $\mathscr{O}$ for
Kac--Moody algebras}

\author{Gurbir Dhillon}
\address[G.~Dhillon]{Department of Mathematics, Yale University,
New Haven, CT 06511, USA}
\email{\tt gurbir.dhillon@yale.edu}

\author{Apoorva Khare}
\address[A.~Khare]{Department of Mathematics, Indian Institute of
Science; and Analysis and Probability Research Group; Bangalore --
560012, India}
\email{\tt khare@iisc.ac.in}

\date{\today}



\begin{abstract}
We give the first positive formulas for the weights of every simple
highest weight module $L(\lambda)$ over an arbitrary Kac--Moody algebra.
Under a mild condition on the highest weight, we also express
the weights of $L(\lambda)$ as an alternating sum similar to the
Weyl--Kac character formula. 

To obtain these results, we show the following data attached to a highest
weight module are equivalent:
(i)~its integrability,
(ii)~the convex hull of its weights,
(iii)~the Weyl group symmetry of its character,
and (iv)~when a localization theorem is available, its behavior on
certain codimension one Schubert cells.
We further determine precisely when the above datum determines
the weights themselves.
Moreover, we use condition (iv)~to
relate localizations of the convex hull of the weights with the
introduction of poles of the corresponding $D$-module on certain
divisors, which answers a question of Brion.

Many of these results are new even in finite type. We prove similar
assertions for highest weight modules over a symmetrizable quantum group.

\end{abstract}
\maketitle

\settocdepth{section}
\tableofcontents

\section{Introduction}\label{Sintro}

A fundamental problem in the representation theory of semisimple and
Kac--Moody Lie algebras is to compute characters of simple highest weight
modules. In all cases where they are known even conjecturally, the
solutions proceed roughly by (i)~expressing simple characters as linear
combinations of Verma characters. The relevant block of Category $\OO$ is
then (ii)~identified as a categorification of a module for the Hecke
algebra. The two bases of simple and Verma characters correspond to two
canonical bases for the module, whose change of basis is determined via
the combinatorics of the Coxeter system.  

The subtlety in steps (i) and (ii) is already apparent for regular
integral blocks for affine Lie algebras. At noncritical levels, as in
finite type, one takes {\em bona fide} Verma modules, and the relevant
Hecke module is the regular representation. At the critical level, one
replaces each Verma module with a quotient, the restricted Verma module,
and the desired Hecke module is the periodic module. To our knowledge,
there is not even a conjectural description of critical blocks beyond
affine type along the lines of (i) and (ii).

Another issue of longstanding interest is to obtain manifestly positive
formulas for simple characters. I.e., in step (i) the coefficients of
Verma characters generically come with signs, as visible in the Weyl--Kac
and Kazhdan--Lusztig character formulas, and one would like alternative
formulae for simple characters without this issue. This problem is solved
via crystal bases for regular dominant and regular anti-dominant highest
weights, but is wide open in general. 

With these motivations, in this paper we address the simpler question of
determining not the character of a simple highest weight module but
rather its weights. I.e., one can ask for the eigenvalues of the Cartan's
action, but not the exact multiplicities. We give multiple formulas for
weights of an arbitrary simple highest weight module for any Kac--Moody
algebra. In light of the above discussion, we emphasize first that the
formulas hold in the cases where there are no predictions for characters
in the spirit of (i) and (ii), and moreover are remarkably uniform across
different levels.  We emphasize second that the formulas are visibly
positive. 

After conversations with experts, it seems the formulas described in this
paper likely did not appear earlier, even as folklore. This is rather
remarkable, considering both the accessibility and naturality of the
question and the transparency of its solution. 


\section{Statements of results}\label{st}

Throughout the paper, unless otherwise specified $\g$ is a Kac--Moody
algebra over $\mathbb{C}$ with fixed triangular decomposition $\g = \n^-
\oplus \h \oplus \n^+$. We will write $I$ for the vertices of its Dynkin
diagram, and denote its simple roots, simple coroots, and the simple
reflections in its Weyl group $W$ by 
\[
\alpha_i, \quad \halpha_i, \quad s_i, \quad \quad \text{for } i \in I,
\]
respectively. We will denote by $V$ a $\g$-module of highest weight
$\lambda \in \h^*$, and by $L(\lambda)$ the simple module with highest
weight $\lambda$.

\subsection{Three positive formulas for the weights of simple
modules}\label{Sbump-results} 

If for $\mu \in \h^*$ we consider the corresponding weight space of
$L(\lambda)$, namely
\[
L(\lambda)_\mu := \{ v \in L(\lambda): hv = \langle \lambda, h \rangle v
\text{ for all } h \in \h \},
\]
then our main result is several descriptions of the {set of}
nonzero weight spaces of $L(\lambda)$, i.e.,
\begin{equation}\label{e:wts}
\wt L(\lambda) := \{ \mu \in \h^*: L(\lambda)_\mu \neq 0 \}.
\end{equation}
In every case when the character of $L(\lambda)$ is known, an essential
role is played by a subgroup of $W$ known as the integral Weyl group. To
determine the weights of $L(\lambda)$, perhaps surprisingly we will need
much less, namely the intersection of the integral Weyl group with the
simple reflections. Accordingly, consider the subset of $I$ given by 
\begin{equation}
I_\lambda := \{ i \in I: \langle \halpha_i, \lambda \rangle \in
\mathbb{Z} \}.
\end{equation}

\subsubsection{}

Our first formula uses the restriction of $L(\lambda)$ to the maximal
standard Levi subalgebra whose action is integrable. To state it, write
$\fl_\lambda$ for the standard Levi subalgebra of $\g$ corresponding to
$I_\lambda$. For $\nu \in \h^*$, write $L_{\hspace{.5mm}
\fl_\lambda}(\nu)$ for the simple $\fl_\lambda$-module of highest weight
$\nu$, and as in \eqref{e:wts} write $\wt L_{\hspace{.5mm}
\fl_\lambda}(\nu)$ for its set of nonzero weights. 

\begin{pro}\label{p1}
{For all $\lambda \in \h^*$, we}
have the equality 
	\begin{equation}
	\wt L(\lambda) = \bigsqcup_{\mu \in \Z^{\geqslant 0} \{ \alpha_i:
	\hspace{1mm} i \in I \setminus I_\lambda \}} \wt
	L_{\mathfrak{l}_\lambda}(\lambda - \mu).
	\end{equation}
\end{pro}

\noindent In words, the proposition asserts that (i)~for every positive
integral combination $\mu$ of the simple roots not indexed by
$I_\lambda$, the weight space $L(\lambda)_{\lambda - \mu}$ is nonzero.
After restriction to $\fl_\lambda$, such a weight space consists of
highest weight vectors, and the proposition implies that (ii)~all the
weights of $L(\lambda)$ may be obtained as union of those of the
$L_{\hspace{.5mm} \fl_\lambda}(\lambda - \mu)$. The latter are integrable
$\fl_\lambda$-modules, and hence their weights are well known, cf.\
Proposition~\ref{nonde} below.

\subsubsection{}

Our second formula shows the relationship between $\wt L(\lambda)$ and
its convex hull, which we denote by $\conv L(\lambda)$. To state it,
recall the standard partial order $\leqslant$ on $\h^*$,
where $\mu \leqslant \lambda$ if $\lambda - \mu$ is a (possibly empty)
sum of simple roots.

\begin{pro}\label{p3}
{For all $\lambda \in \h^*$, we}
have the equality
	\begin{equation} \label{e:p3}
	\wt L(\lambda) = \conv L(\lambda) \cap \{\mu \in \h^*: \mu
	\leqslant \lambda \}.
	\end{equation} 
\end{pro}

\noindent In words, the proposition asserts that the left-hand side of
Equation \eqref{e:p3}, which necessarily lies in the intersection of the
two terms of the right-hand side, in fact coincides with it. The question
of whether Proposition~\ref{p3} holds, i.e., whether the weights of
simple highest weight modules are no finer an invariant than their convex
hull, was brought to our attention by Daniel Bump~\cite{Bump}. 

Moreover, we give an explicit description of $\conv L(\lambda)$, and more
generally of the convex hull of the weights of any highest weight module,
in Proposition~\ref{ray} below. 

\subsubsection{}

Our third formula describes the weights of $L(\lambda)$ using the group
{$W_{I_\lambda}$} generated by the simple reflections $s_i,$ for $i \in
I_\lambda$, i.e., the Weyl group of $\fl_\lambda$. We will also need the
dominant integral weights of $\fl_\lambda$, namely
\[
P^+_\lambda := \{ \mu \in \h^*: \langle \mu, \halpha_i \rangle \in
\mathbb{Z}^{\geqslant 0}, i \in I_\lambda \}.
\]
\begin{pro}\label{p2}
	If $\lambda$ has finite stabilizer in $W_{I_\lambda}$, then we
	have the equality
	\begin{equation} \label{e:p2}
	\wt L(\lambda) = {W_{I_\lambda}} \{ \mu \in P^+_\lambda: \mu
	\leqslant \lambda \}.
	\end{equation}
\end{pro}

\noindent We remind that the assumption that $\lambda$ has finite
stabilizer is very mild, as recalled in Proposition \ref{tits}(4). In
words, the proposition states that the left-hand side of \eqref{e:p2},
which necessarily lies in the right-hand side by the integrability of
$L(\lambda)$ as a $\fl_\lambda$-module, coincides with it.

\subsubsection{}

Propositions \ref{p1}, \ref{p3}, and \ref{p2} are well known for
integrable $L(\lambda)$, and Propositions \ref{p1} and \ref{p3} were
proved by the second named author in finite type \cite{Kh1}. The
remaining cases, in particular Proposition \ref{p2} in all types, are to
our knowledge new. 

Propositions \ref{p1}, \ref{p3}, and \ref{p2} are particularly striking
in infinite type. When specialized to affine Kac--Moody algebras, the
formulae are insensitive to whether the highest weight is critical or
non-critical. In contrast, critical level modules behave very differently
from noncritical modules, even at the level of characters. For
symmetrizable Kac--Moody algebras, we similarly obtain weight formulae
for highest weights on the critical hyperplanes. The authors are unaware
of even conjectural formulae for the simple characters in this case.
Finally, for non-symmetrizable Kac--Moody algebras, we recall that it is
unknown how to compute weight space multiplicities even for integrable
$L(\lambda)$. 

\subsection{Three invariants of highest weight modules}

The main input into our three formulas presented above is the equivalence
of three invariants of general highest weight modules. In what follows,
we denote the simple lowering operators by $f_i,$ for  $i \in I$.

\begin{theo}\label{maintheo}
The following invariants of a highest weight module $V$ determine one
another.
\begin{enumerate}
\item The {\em integrability} of $V$, i.e. the subset of the vertices of
the Dynkin diagram defined as
\[
I_V := \{i \in I: f_i \text{ acts locally nilpotently on } V \}.
\]

\item The convex hull of the weights of $V$.

\item The stabilizer of the character of $V$ in the Weyl group $W$.
\end{enumerate}

\noindent Explicitly, the convex hull in~(2) is always that of the
parabolic Verma module $M(\lambda, I_V)$ (cf. Equation \eqref{Epvm}
below), and the stabilizer in~(3) is always the parabolic subgroup
$W_{I_V}$. 
\end{theo}

\noindent To our knowledge, the equivalences of Theorem \ref{maintheo}
are new even for $\g$ of finite type. Moreover, when a
localization theorem is available, we describe an equivalent geometric
invariant of the perverse sheaf corresponding to $V$ in Proposition
\ref{gint} below. 

In the proof of Theorem \ref{maintheo}, the difficult implication is that
the natural map 
\[
M(\lambda, I_V) \otimes V_\lambda \rightarrow V,
\]
where $V_\lambda$ is the highest weight line of $V$, induces
an equality on convex hulls of weights.

For a general highest weight module $V$, its set of weights is a finer
invariant than the datum of Theorem \ref{maintheo}. For example, for $\g
= \mathfrak{sl}_2 \oplus \mathfrak{sl}_2$, the highest weight modules
$M(0)$ and $V = M(0)/M(s_1s_2 \cdot 0)$, where $s_1, s_2$ are the simple
reflections and $\cdot$ denotes the dot action, have different sets of
weights but the same convex hull. In the following theorem, we determine
precisely when the two invariants are equivalent. 

To state it, we consider the {\em potential integrability} of $V$, i.e.
the subset of the vertices of the Dynkin diagram given by
\[
 I_\lambda \setminus I_V.
\]
Note these index the simple roots for which the corresponding
$\mathfrak{sl}_2$ action is integrable on a nonzero quotient of $V$. The
following theorem shows that the phenomenon of the above example, namely
the potential integrability contains orthogonal roots, is the only
obstruction to $\conv V$ and $\wt V$ being distinct invariants. 

\begin{theo}\label{Tnoholes}
Fix $\lambda \in \h^*$ and $\mathring{I} \subset I_\lambda$. Every
highest weight module $V$ of highest weight $\lambda$ and integrability
$\mathring{I}$ has the same weights if and only if $I_\lambda \setminus
\mathring{I}$ is complete, i.e.,
\[
\langle \halpha_i, \alpha_j \rangle \neq 0, \quad \text{for all } i,j \in
I_\lambda \setminus \mathring{I}.
\]
\end{theo}

The question of whether Theorem \ref{Tnoholes} holds was brought to our
attention by James Lepowsky \cite{Lepowsky}. We are not aware of a
precursor to Theorem \ref{Tnoholes} in the literature, even in finite
type.


\subsection{An alternating formula for the weights of simple
modules}\label{S251}

We now discuss an alternating formula for the weights of $L(\lambda)$. We
will first write the formula, and then explain the appearing notation. 

\begin{theo}\label{wwks}
For $\lambda \in \h^*$ such that the stabilizer of $\lambda$ in
{$W_{I_\lambda}$} is finite, we have:
\begin{equation}\label{wks}
\wt L(\lambda) = \sum_{w \in {W_{I_\lambda}} } w
\frac{e^\lambda}{\prod_{i \in I} (1 - e^{-\alpha_i})}.
\end{equation}
\end{theo}

\noindent On the left-hand side of \eqref{wks} we mean the
`multiplicity-free'
character 
\[
\sum_{\mu \in \h^*: L(\lambda)_\mu \neq 0} e^\mu.
\]
On the right-hand side of \eqref{wks}, in each summand we take the
`highest weight' expansion of $w(1 - e^{-\alpha_i})^{-1}$, i.e.:
\[
w \frac{1}{1 - e^{-\alpha_i}} := \begin{cases}
1 + e^{-w \alpha_i} + e^{-2 w \alpha_i} + \cdots, & \text{ if } w
\alpha_i > 0, \\
- e^{ w\alpha_i} - e^{2w \alpha_i} - e^{3w\alpha_i} - \cdots, & \text{ if
} w \alpha_i < 0.
\end{cases}
\]

We now make three qualitative comments on the appearance of \eqref{wks}.
First, as claimed in the introduction, each summand comes with signs but
no multiplicities as in Kazhdan--Lusztig theory. Second, we note that the
denominator in \eqref{wks} is not the usual Weyl denominator, but instead
a multiplicity-free variant which runs only over the simple roots, i.e.
replaces $\n$ by $\n/[\n, \n]$. Finally, the group
{$W_{I_\lambda}$} which indexes the sum is a parabolic
subgroup of $W$, rather than the integral Weyl group which appears in the
known cases of the characters of simple highest weight modules. 

Theorem \ref{wwks} was previously known for integrable modules by work of
Kass \cite{Kass}, Brion \cite{Brion}, Walton \cite{Walton2}, Postnikov
\cite{Postnikov}, and Sch\"utzer \cite{Sc}, where it closely resembles
the usual Weyl--Kac character formula \cite{Weyl}. As pointed out to us
by Michel Brion~\cite{Brion2}, in finite type Theorem \ref{wwks} follows
from a more general formula for exponential sums over polyhedra, cf.\
Remark~\ref{otherone} below. All other cases are to our knowledge new.

\subsection{A geometric interpretation of integrability}\label{Sbrion2}

In this section we use a topological manifestation of the datum in
Theorem \ref{maintheo} to answer a question brought to our attention by
Brion~\cite{Brion2} in geometric representation theory. 

To pose {the question}, let $\g$ be of finite type, $G$ the
corresponding simply-connected algebraic group, and $B$ the subgroup with
Lie algebra $\h \oplus \n^+$. Fix $\lambda$ a regular dominant integral
weight, and consider $\conv L(\lambda)$, i.e., the Weyl polytope. For any
$\mathring{I} \subset I$, if one intersects the tangent cones of $\conv
L(\lambda)$ at the vertices 
\[
w \lambda, \quad \text{for } w \in W_{\mathring{I}},
\]
one obtains the convex hull of $M(\lambda, \mathring{I})$. Brion first
asked whether an analogous formula holds for other highest weight
modules. We answer this question affirmatively in the companion paper
\cite{DK2}. 

Brion further observed that this procedure of localizing $\conv
L(\lambda)$ in convex geometry is the shadow of a localization in complex
geometry. More precisely, consider the flag variety $G/B$, and let
$\mathscr{L}_\lambda$ denote the line bundle with 
\[
H^0(G/B,
\mathscr{L}_\lambda) \simeq L(\lambda).
\]
As usual, for any $w \in W$ write $X_w$ for the closure of the Bruhat
cell $BwB/B$, and write $w_\circ$ for the longest element of $W$. Let
$\mathbb{D}$ denote the standard dualities on Category $\mathscr{O}$ and
regular holonomic $D$-modules, i.e. the contragredient and Verdier
dualities, respectively. Then for a union of Schubert divisors 
\[
Z = \bigcup_{i \in I \setminus \mathring{I}}
X_{s_i w_\circ}
\]
with complement $U$, we have 
\[
H^0(U, \mathscr{L}_\lambda) \simeq \mathbb{D} M(\lambda, \mathring{I}).
\]
Thus in this case localization of the convex hull could be recovered as
taking the convex hull of the weights of sections of
$\mathscr{L}_\lambda$ on an appropriate open set $U$. Brion raised in
\cite{Brion2} the question of whether a similar result holds for more
general highest weight modules. 

We answer this affirmatively for a regular integral infinitesimal
character. By translation, it suffices to examine the regular block
$\mathscr{O}_0$.

\begin{theo}\label{brgg}
Let $\g$ be of finite type, and $\lambda = w \cdot - 2\rho$. Let $V$ be a
$\g$-module of highest weight $\lambda$, and write $\cV$ for the
corresponding $D$-module on $G/B$. For $\mathring{I} \subset I_V$, set 
\[ Z = \bigcup_{i
\in I_V \setminus \mathring{I}}  X_{s_i w},\]
and write $U$ for its open complement in $G/B$. Then the $\g$-module
\[\mathbb{D} H^0(U, \mathbb{D}\cV)\]
has highest weight $\lambda$ and integrability $\mathring{I}$.
\end{theo}

Let us mention one ingredient in the proof of Theorem \ref{brgg}, namely
that the datum of Theorem \ref{maintheo} has the following geometric
manifestation. 

\begin{pro}\label{gint}
For $V$ and $\mathscr{V}$ as above, we have
\[
I_V = \{ i \in I: \text{the {$*$}-restriction of $\mathscr{V}$ to
$C_{s_iw}$ is nonzero} \}.
\]
\end{pro}

\subsection{Highest weight modules over symmetrizable quantum groups}

In the final section, we apply both the methods and the results from
earlier to study highest weight modules over quantum groups, and obtain
results similar to those discussed above.

\subsection{A further problem -- multiplicity-free Macdonald identities
}\label{Sfuture}

We conclude this section by calling attention to a problem suggested by
our results. Namely, the statement of Theorem \ref{wwks} includes the
assumption that the highest weight has finite integrable stabilizer.
However, the asserted identity otherwise tends to fail in interesting
ways. For example, we found the following by a direct calculation (see
also the extended abstract \cite{DK-fpsac17}).

\begin{pro}\label{Pweylkac}
For $\g$ of rank $2$ and the trivial module $L(0)$, we have
\begin{equation}\label{Eimaginary}
\sum_{w \in W} w \frac{1}{(1 - e^{-\alpha_1})(1 - e^{-\alpha_2})} = 1 +
\sum_{\alpha \in \Delta_-^{im}}e^\alpha,
\end{equation}
where $\Delta_-^{im}$ denotes the set of negative imaginary roots.
\end{pro} 

\noindent It would be interesting to extend Theorem \ref{wwks} to more
general objects, e.g. the trivial representation, which may be thought of
as multiplicity-free Macdonald identities. Indeed, Equation
\eqref{Eimaginary} suggests correction terms coming from imaginary roots,
akin to \cite{Kac74}.

\begin{re}
Since the writing of this paper, such a multiplicity-free Macdonald
identity was established in untwisted affine type by Niu--Zhang
\cite{SPUR}.
\end{re} 

\subsection*{Organization of the paper}

After introducing notation in Section \ref{not}, we prove in Section
\ref{S4} the main result. In Sections \ref{Sbump}--\ref{Sgeometry}, we
develop different applications of the main result. Finally, in Section
\ref{Squantum}, we discuss extensions of the previous sections to
quantized enveloping algebras.

\subsection*{Acknowledgments}

It is a pleasure to thank Asilata Bapat, Michel Brion, Daniel Bump,
Galyna Dobrovolska, Ian Grojnowski, Shrawan Kumar, James Lepowsky, Sam
Raskin, Geordie Williamson, and Zhiwei Yun for valuable discussions.
The work of G.D.~is partially supported by the Department of Defense
(DoD) through the NDSEG fellowship.
A.K.~was partially supported by Ramanujan Fellowship grant
SB/S2/RJN-121/2017, MATRICS grant MTR/2017/000295, and SwarnaJayanti
Fellowship grants SB/SJF/2019-20/14 and DST/SJF/MS/2019/3 from SERB and
DST (Govt.~of India), by grant F.510/25/CAS-II/2018(SAP-I) from UGC
(Govt.~of India), and by a Young Investigator Award from the Infosys
Foundation.

\section{Preliminaries and notations}\label{not}

The contents of this section are mostly standard. We advise the reader to
skim Subsection \ref{wwey}, and refer back to the rest only as needed.

\subsection{Notation for numbers and sums}

We write $\Z$ for the integers, and $\R, \C$ for the real
and complex numbers respectively. For a subset $S$ of a real vector space
$E$, we write $\Z^{\geqslant 0} S$ for the set of finite linear
combinations of $S$ with coefficients in $\Z^{\geqslant 0}$, and
similarly $\Z S, \R^{\geqslant 0} S$, etc.

\subsection{Notation for Kac--Moody algebras, standard parabolic and Levi
subalgebras}\label{rat}

The basic references are \cite{Kac} and \cite{Kumar}.
In this paper we work throughout over $\C$. Let $I$ be a finite set,
and $A = (a_{ij})_{i,j \in I}$ a generalized Cartan matrix.
Fix a realization $(\h, \pi, \hpi)$, with simple roots $\pi = \{\alpha_i
\}_{i \in I} \subset \h^*$ and coroots $\hpi = \{\halpha_i \}_{i \in I}
\subset \h$ satisfying $(\halpha_i, \alpha_j) = a_{ij}, \forall i,j \in
I$.

Let $\g := \g(A)$ be the associated Kac--Moody algebra generated by $\{
e_i, f_i : i \in I \}$ and $\h$, modulo the relations:
\begin{align*}
[e_i, f_j] = &\ \delta_{ij} \halpha_i, \
[h, e_i] = (h, \alpha_i) e_i, \
[h, f_i] = - (h, \alpha_i) f_i, \
[\h,\h] = 0, \quad \forall h \in \h, \ i,j \in I,\\
& ({\rm ad}\ e_i)^{1 - a_{ij}}(e_j) = 0, \quad
({\rm ad}\ f_i)^{1 - a_{ij}}(f_j) = 0, \quad \forall i,j \in I,
\ i \neq j.
\end{align*}

Denote by $\overline{\g}(A)$ the quotient of $\g(A)$ by the largest ideal
intersecting $\h$ trivially; these coincide when $A$ is symmetrizable.
When $A$ is clear from context, we will abbreviate these to $\g,
\overline{\g}$.

In the following we establish notation for $\g$; the same apply for
$\overline{\g}$ \textit{mutatis mutandis}.
Let $\Delta^+, \Delta^-$ denote the sets of positive and negative roots,
respectively.
We write $\alpha > 0$ for $\alpha \in \Delta^+$, and
similarly $\alpha < 0$ for $\alpha \in \Delta^-$. For a sum of roots
$\beta = \sum_{i \in I} k_i \alpha_i$ with all $k_i \geqslant 0$, write
$\supp \beta := \{ i \in I: k_i \neq 0 \}$. Write
\[
\n^- := \bigoplus_{ \alpha < 0} \g_\alpha, \qquad
\n^+ := \bigoplus_{\alpha > 0 } \g_{\alpha}.
\]

Let $\leqslant$ denote the standard partial order on $\h^*$, i.e.~for
$\mu, \lambda \in \h^*$, $\mu \leqslant \lambda$ if and only if $\lambda
- \mu  \in \Z^{\geqslant 0} \pi$. 

For any $J \subset I$, let $\fl_J$ denote the associated Levi subalgebra
generated by $\{ e_i, f_i : i \in J \}$ and $\h$. For $\lambda \in \h^*$
write $L_{\hspace{.5mm} \fl_J}(\lambda)$ for the simple $\fl_J$-module of
highest weight $\lambda$. Writing $A_J$ for the principal submatrix
$(a_{i,j})_{i,j \in J}$, we may (non-canonically) realize $\g(A_J) =:
\g_J$ as a subalgebra of $\g(A)$. Now write $\pi_J, \Delta^+_J,
\Delta^-_J$ for the simple, positive, and negative roots of $\g(A_J)$ in
$\h^*$, respectively (note these are independent of the choice of
realization).
Finally, we define the associated Lie subalgebras $\mathfrak{u}^+_J,
\mathfrak{u}^-_J, \n^+_J, \n^-_J$ by:
\[
\mathfrak{u}^\pm_J := \bigoplus_{\alpha \in \Delta^\pm \setminus
\Delta^\pm_J} \g_\alpha, \qquad \n^\pm_J := \bigoplus_{\alpha \in
\Delta^\pm_J} \g_\alpha,
\]

\noindent and $\mathfrak{p}_J := \fl_J \oplus \mathfrak{u}^+_J$ to be the
associated parabolic subalgebra.

\subsection{Weyl group, parabolic subgroups, Tits cone}\label{wwey}

Write $W$ for the Weyl group of $\g$, generated by the simple reflections
$\{ s_i, i \in I \}$, and let $\ell: W \rightarrow \Z^{\geqslant 0}$ be
the associated length function. For $J \subset I$, let $W_J$ denote the
parabolic subgroup of $W$ generated by $\{ s_j, j \in J \}$.

Write $P^+$ for the dominant integral weights, i.e.~$\{\mu \in \h^*:
(\halpha_i, \mu) \in \Z^{\geqslant 0}, \forall i \in I \}$. The following
choice is non-standard. Define the real subspace $\h^*_\R := \{ \mu \in
\h^*: (\halpha_i, \mu) \in \R, \forall i \in I \}$. Now define the {\em
dominant chamber} as $D := \{ \mu \in \h^*: (\halpha_i, \mu) \in
\R^{\geqslant 0}, \forall i \in I \} \subset \h^*_\R$, and the {\em Tits
cone} as $C := \bigcup_{w \in W} wD$. 

\begin{re}
In \cite{Kumar} and \cite{Kac}, the authors define $\h^*_\R$ to be a real
form of $\h^*$. This is smaller than our definition whenever the
generalized Cartan matrix $A$ is non-invertible, and has the consequence
that the dominant integral weights are not all in the dominant chamber,
unlike for us. This is a superficial difference, but our convention helps
avoid constantly introducing arguments like \cite[Lemma 8.3.2]{Kumar}.
\end{re}

We will also need parabolic analogues of the above. For $J \subset I$,
define $\h^*_\R(J) := \{ \mu \in \h^*: (\halpha_j, \mu) \in \R, \forall j
\in J \}$, the $J$ {\em dominant chamber} as $D_J := \{\mu \in \h^*:
(\halpha_j, \mu) \in \R^{\geqslant 0}, \forall j \in J \}$, and the $J$
{\em Tits cone} as $C_J := \bigcup_{w \in W_J} wD_J$. Finally, we write
$P^+_J$ for the $J$ dominant integral weights, i.e.~$\{ \mu \in \h^*:
(\halpha_j, \mu) \in \Z^{\geqslant 0}, \forall j \in J\}.$ The following
standard properties will be used without further reference in the paper:

\begin{pro}\label{tits}
For $\g(A)$ with realization $(\h, \pi, \hpi)$, let $\g(A^t)$ be the dual
algebra with realization $(\h^*, \hpi, \pi)$. Write $\check{\Delta}^+_J$
for the positive roots of the standard Levi subalgebra $\fl_J^t \subset
\g(A^t)$. 
\begin{enumerate}
\item For $\mu \in D_J$, the isotropy group $\{w \in W_J: w\mu = \mu \}$
is generated by the simple reflections it contains.

\item The $J$ dominant chamber is a fundamental domain for the action of
$W_J$ on the $J$ Tits cone, i.e., every $W_J$ orbit in $C_J$ meets $D_J$
in precisely one point. 

\item $C_J = \{ \mu \in \h^*_\R(J): (\halpha, \mu) < 0 \text{ for at most
finitely many } \halpha \in \check{\Delta}^+_J \}$; in particular, $C_J$
is a convex cone.

\item Consider $C_J$ as a subset of  $\h^*_\R(J)$ in the analytic
topology, and fix $\mu \in C_J$. Then $\mu$ is an interior point of $C_J$
if and only if the isotropy group $\{w \in W_J: w\mu = \mu \}$ is finite. 
\end{enumerate}
\end{pro}

\begin{proof}
The reader can easily check that the standard arguments, cf.\
\cite[Proposition 1.4.2]{Kumar}, apply in this setting {\em mutatis
mutandis}.
\end{proof}

We also fix $\rho \in \h^*$ satisfying $(\halpha_i, \rho) = 1, \forall i
\in I$, and define the {\em dot action} of $W$ via $w \cdot \mu := w(\mu
+ \rho) - \rho$; this does not depend on the choice of $\rho$.

\subsection{Representations, integrability, and parabolic Verma
modules}\label{weyl}

Given an $\h$-module $M$ and $\mu \in \h^*$, write $M_\mu$ for the
corresponding simple eigenspace of $M$, i.e.~$M_\mu := \{ m \in M : hm =
(h,\mu) m\ \forall h \in \h \}$, and write $\wt M := \{ \mu \in \h^* :
M_\mu \neq 0 \}$.

Let $V$ be a highest weight $\g$-module with highest weight $\lambda \in
\h^*$. For $J \subset I$, we say $V$ is {\em $J$ integrable} if $f_j$
acts locally nilpotently on $V$, $\forall j \in J$. The following
standard lemma may be deduced from \cite[Lemma 1.3.3]{Kumar} and the
proof of \cite[Lemma 1.3.5]{Kumar}.

\begin{lem}\label{wint}\hfill
\begin{enumerate}
\item If $V$ is of highest weight $\lambda$, then $V$ is $J$ integrable
if and only if $f_j$ acts nilpotently on the highest weight line
$V_\lambda$, $\forall j \in J$. 

\item If $V$ is $J$ integrable, then $\ch V$ is $W_J$-invariant. 
\end{enumerate}
\end{lem}

Let $I_V$ denote the maximal $J$ for which $V$ is $J$ integrable, i.e.,
\begin{equation}
I_V = \{ i \in I : (\halpha_i, \lambda) \in \Z^{\geqslant 0}, f_i^{
(\halpha_i, \lambda) + 1} V_\lambda = 0 \}.
\end{equation}
We will call $W_{I_V}$ the {\em integrable Weyl group}. 

We next remind the basic properties of parabolic Verma modules over
Kac--Moody algebras. These are also known in the literature as
generalized Verma modules, e.g.~in the original papers by Lepowsky (see
\cite{lepo1} and the references therein).

Fix $\lambda \in \h^*$ and a subset $J$ of $I_\lambda = \{ i \in
I: (\halpha_i, \lambda) \in \Z^{\geqslant 0} \}$. The parabolic Verma
module $M(\lambda, J)$ co-represents the following functor from $\g\Mod$
to Set:
\begin{equation}\label{Epvm}
M \rightsquigarrow \{ m \in M_\lambda : \n^+ m = 0,\ f_j \text{ acts
nilpotently on } m, \forall j \in J \}.
\end{equation}

\noindent When $J$ is empty, we simply write $M(\lambda)$ for the Verma
module.
From the definition and Lemma \ref{wint}, it follows that
$M(\lambda, J)$ is a highest weight module that is $J$ integrable.

More generally, for any subalgebra $\fl_J \subset \fs \subset \g$, the
module $M_\fs(\lambda,J)$ will co-represent the functor from $\fs\Mod$ to
Set:
\begin{equation}\label{Ecorep}
M \rightsquigarrow \{ m \in M_\lambda : (\n^+ \cap \fs) m = 0,\ f_j
\text{ acts nilpotently on } m, \forall j \in J \}.
\end{equation}

\noindent We will only be concerned with $\fs$ equal to a Levi or
parabolic subalgebra. In accordance with the literature, in the
case of ``full integrability'' we define
$L^{\max}(\lambda) := M(\lambda,I)$ for $\lambda \in P^+$, and similarly
$L^{\max}_{\mathfrak{p}_J}(\lambda) := M_{\mathfrak{p}_J}(\lambda,J)$ and
$L^{\max}_{\fl_J}(\lambda) := M_{\fl_J}(\lambda,J)$ for $\lambda \in
P_J^+$. Finally, write $L_{\hspace{.5mm} \fl_J}(\lambda)$ for the simple
quotient of $L_{\hspace{.5mm} \fl_J}^{\max}(\lambda).$ 

\begin{pro}[Basic formulae for parabolic Verma modules]\label{para}
Let $\lambda \in P^+_J$.
\begin{enumerate}
\item $M(\lambda, J) \simeq M(\lambda)/( f_j^{(\halpha_j, \lambda) + 1}
M(\lambda)_\lambda, \forall j \in J).$
 
\item Given a Lie subalgebra $\fl_J \subset \fs \subset \g$, define
$\fs_J^+ := \fl_J \oplus (\mathfrak{u}_J^+ \cap \fs)$. Then:
\begin{equation}
M_{\fs}(\lambda, J) \simeq
\Ind^\fs_{\fs_J^+} \Res_{\fs_J^+}^{\fl_J} L^{\max}_{\fl_J}(\lambda),
\end{equation}

\noindent where we view $\fl_J$ as a quotient of $\fs_J^+$ via the short
exact sequence:
\[
0 \to \mathfrak{u}_J^+ \cap \fs \to \fs_J^+ \to \fl_J \to 0.
\]
In particular, for $\fs = \g$:
\begin{equation}\label{Everma}
\wt M(\lambda, J) = \wt L_{\hspace{.5mm} \fl_J}(\lambda) - \Z^{\geqslant
0} (\Delta^+ \setminus \Delta^+_J).
\end{equation}
\end{enumerate}
\end{pro}

We remark that some authors use the term inflation (from $\fl_J$ to
$\fs_J^+$) in place of $\Res^{\fl_J}_{\fs_J^+}$.

\section{Three invariants of highest weight modules}\label{S4}

In this section we prove the main result of the paper, Theorem
\ref{maintheo}, except for the perverse sheaf formulation, which we
address in Section \ref{Sgeometry}. The two tools we use are:
(i) an ``Integrable Slice Decomposition'' of the weights of parabolic
Verma modules $M(\lambda,J)$, which extends a previous construction in
\cite{Kh1} to representations of Kac--Moody algebras; and
(ii) a ``Ray Decomposition'' of the convex hull of these weights, which
is novel in both finite and infinite type for non-integrable highest
weight modules.

\subsection{The Integrable Slice Decomposition}

Our first goal will be to prove the following statement.

\begin{pro}[Integrable Slice Decomposition]\label{slice}
\begin{equation}\label{Eslice}
\wt M(\lambda, J) = \bigsqcup_{\mu \in \Z^{\geqslant 0} (\pi \setminus
\pi_J) } \wt L_{\hspace{.5mm} \fl_J}(\lambda - \mu).
\end{equation}
In particular, $\wt M(\lambda, J)$ lies in the $J$ Tits cone (cf.\
Section~\ref{wwey}).
\end{pro}


In proving Proposition \ref{slice} and below, the following results will
be of use to us: 

\begin{pro}[{\cite[\S 11.2 and Proposition 11.3(a)]{Kac}}]\label{nonde}
For $\lambda \in P^+, \ \mu \in \h^*$, say $\mu$ is {\em non-degenerate
with respect to $\lambda$} if $\mu \leqslant \lambda$ and $\lambda$ is
not perpendicular to any connected component of $\supp (\lambda - \mu)$,
cf.\ Section \ref{rat} for notation.
Let $V$ be an integrable module of highest weight $\lambda$.
\begin{enumerate}
\item If $\mu \in P^+$, then $\mu \in \wt V$ if and only if $\mu$ is
non-degenerate with respect to $\lambda$. 

\item If the sub-diagram on $\{ i \in I: (\halpha_i, \lambda) = 0 \}$ is
a disjoint union of diagrams of finite type, then $\mu \in P^+$ is
non-degenerate with respect to $\lambda$ if and only if $\mu \leqslant
\lambda$.

\item $\wt V = (\lambda - \Z^{\geqslant 0} \pi) \cap \conv(W \lambda)$.
\end{enumerate}
\end{pro}

Proposition \ref{nonde} is explicitly stated in \cite{Kac} for
$\overline{\g}$, but also holds for $\g$. To see this, note that
$\wt V \subset (\lambda - \Z^{\geqslant 0} \pi) \cap \conv(W \lambda)$,
and the latter are the weights of its simple quotient $L(\lambda)$, which
is inflated from $\overline{\g}$.

\begin{proof}[Proof of Proposition \ref{slice}]
The disjointness of the terms on the right-hand side is an easy
consequence of the linear independence of simple roots. 

We first show the inclusion $\supset$. Recall the isomorphism of
Proposition \ref{para}:
\[
M(\lambda, J) \simeq M(\lambda)/ (f_j^{(\halpha_j, \lambda) + 1}
M(\lambda)_{\lambda}, \forall j \in J).
\]

It follows that the weights of $\ker( M(\lambda) \to M(\lambda, J))$ are
contained in $\bigcup_{j \in J} \{ \nu \leqslant s_j \cdot \lambda \}$.
Hence $\lambda - \mu$ is a weight of $M(\lambda, J), \forall \mu \in
\Z^{\geqslant 0} \pi_{I \setminus J}$. Any nonzero element of $M(\lambda,
J)_{\lambda - \mu}$ generates an integrable highest weight
$\fl_J$-module. As the weights of all such modules coincide by
~Proposition \ref{nonde}(3), this shows the inclusion $\supset$.

We next show the inclusion $\subset$.
For any $\mu \in \Z^{\geqslant 0} \pi_{I \setminus J}$, the `integrable
slice'
\[
S_{\mu} := \bigoplus_{\nu \in \lambda - \mu + \Z \pi_J}
M(\lambda, J)_{\nu}
\]
lies in Category $\mathscr{O}$ for $\fl_J$, and is furthermore an
integrable $\fl_J$-module. It follows that the weights of $M(\lambda, J)$
lie in the $J$ Tits cone.

Let $\nu$ be a weight of $M(\lambda, J)$, and write $\lambda - \nu =
\mu_{J} + \mu_{I \setminus J}, \mu_J \in \Z^{\geqslant 0} \pi_J, \mu_{I
\setminus J} \in \Z^{\geqslant 0} \pi_{I \setminus J}$. We need to show
$\nu \in \wt L_{\hspace{.5mm} \fl_J}(\lambda - \mu_{I \setminus J})$. By
$W_J$-invariance,  we may assume $\nu$ is $J$ dominant. By Proposition
\ref{nonde}, it suffices to show that $\nu$ is non-degenerate with
respect to $(\lambda - \mu_{I \setminus J})$. To see this, using
Proposition \ref{para} write: 
\[
\nu = \lambda - \mu_L - \sum_{k=1}^n \beta_k, \quad \text{where} \quad
\lambda - \mu_L \in \wt L_{\hspace{.5mm} \fl_J}(\lambda),\ \beta_k \in
\Delta^+ \setminus \Delta^+_J,\ 1 \leqslant k \leqslant n.
\]
The claimed nondegeneracy follows from the fact that $\lambda - \mu_L$ is
nondegenerate with respect to $\lambda$ and that the support of each
$\beta_k$ is connected.
\end{proof}

As an immediate consequence of the Integrable Slice Decomposition
\ref{slice}, we present a family of decompositions of $\wt M(\lambda,J)$,
which interpolates between the two sides of Equation \eqref{Eslice}:

\begin{cor}\label{Cslice}
For subsets $J \subset J' \subset I$, we have:
\begin{equation}
\wt M(\lambda, J) = \bigsqcup_{\mu \in \Z^{\geqslant 0} (\pi \setminus
\pi_{J'})} \wt M_{\fl_{J'}}(\lambda - \mu, J),
\end{equation}

\noindent where $M_{\fl_{J'}}(\lambda, J)$ was defined in Equation
\eqref{Ecorep}.
\end{cor}

As a second consequence, Proposition \ref{nonde}(2) and the Integrable
Slice Decomposition \ref{slice} yield the following simple description of
the weights of most parabolic Verma modules.

\begin{cor}\label{Cfinite}
Suppose $\lambda$ has finite $W_J$-isotropy. Then,
\begin{equation}
\wt M(\lambda, J) = \bigcup_{w \in W_J} w \{ \nu \in P^+_J : \nu
\leqslant \lambda \}.
\end{equation}
\end{cor}

Equipped with the Integrable Slice Decomposition, we provide a
characterization of the weights of a parabolic Verma module that will be
helpful in Section \ref{Sbump}.

\begin{pro}\label{noholes}
\begin{equation}
\wt M(\lambda, J) = (\lambda + \Z \pi) \cap \conv M(\lambda, J).
\end{equation}
\end{pro}

\begin{proof}
The inclusion $\subset$ is immediate. For the reverse $\supset$, note
that $\conv M(\lambda, J)$ lies in the $J$ Tits cone and is $W_J$
invariant. It then suffices to consider a point $\nu$ of the right-hand
side which is $J$ dominant. Write $\nu$ as a convex combination:
\[
\nu = \sum_{k=1}^n t_k (\lambda - \mu^k_{I \setminus J} - \mu^k_{J}),
\]
where $t_k \in \R^{> 0}, \sum_k t_k = 1, \mu^k_{I \setminus J} \in
\Z^{\geqslant 0} \pi_{I \setminus J}, \mu^k_J \in \Z^{\geqslant 0} \pi_J,
1 \leqslant k \leqslant n$. By Propositions \ref{nonde} and \ref{slice},
it remains to observe that $\nu$ is non-degenerate with respect to
$\lambda - \sum_k t_k \mu^k_{I \setminus J}$, as a similar statement
holds for each summand.
\end{proof}

\subsection{The Ray Decomposition}

Using the Integrable Slice Decomposition, we obtain the following novel
description of $\conv M(\lambda, J)$, which is of use in the proof of the
main theorem and throughout the paper.

\begin{pro}[Ray Decomposition]\label{ray}
\begin{equation}
\conv M(\lambda, J) = \conv \bigcup_{w \in W_J,\ i \in I \setminus J} w
(\lambda - \Z^{\geqslant 0} \alpha_i).
\end{equation}
When $J = I$, by the right-hand side we mean $\conv \bigcup_{w \in W} w
\lambda$. 
\end{pro}

\begin{proof}
The containment $\supset$ is straightforward using the definition of
integrability and $W_J$ invariance. For the containment $\subset$, it
suffices to show every weight of $M(\lambda, J)$ lies in the right-hand
side. Note that the right-hand side contains $W_J(\lambda - \Z^{\geqslant
0} \pi_{I \setminus J})$, so we are done by the Integrable Slice
Decomposition \ref{slice}.
\end{proof}


\subsection{Proof of the main result}

With the Integrable Slice and Ray Decompositions in hand, we now turn to
our main theorem.

\begin{theo}\label{intgrp}
Given a highest weight module $V$ and a subset $J \subset I$, the
following statements are equivalent.
\begin{enumerate}
\item One has the equality of subsets of vertices of the Dynkin diagram
$J = I_V$.

\item One has the equality of convex hulls $\conv \wt V = \conv \wt
M(\lambda, J)$.

\item The stabilizer of $\conv V$ in $W$ is $W_J$.
\end{enumerate}
\end{theo}

\noindent The connection to perverse sheaves promised in Section
\ref{Sintro} is proven in Proposition \ref{poles}.

\begin{proof}
To show (1) implies (2), note that $\lambda - \Z^{\geqslant 0} \alpha_i
\subset \wt V, \forall i \in \pi_{I \setminus I_V}$. The implication now
follows from the Ray Decomposition \ref{ray}.
The remaining implications follow from the assertion that the stabilizer
of $\conv V$ in $W$ is $W_{I_V}$. Thus, it remains to prove the
assertion. It is standard that $W_{I_V}$ preserves $\conv V$. For the
reverse, since (1) implies (2), we may reduce to the case of $V =
M(\lambda, J)$.
Suppose $w \in W$ stabilizes $\conv M(\lambda, J)$. It is easy to see
$\lambda$ is a face of $\conv M(\lambda,J)$, hence so is $w \lambda$.
However by the Ray Decomposition \ref{ray}, it is clear that the only
$0$-faces of $\conv M(\lambda,J)$ are $W_{J} (\lambda)$, so without loss
of generality we may assume $w$ stabilizes $\lambda$. 

Recalling that $W_{J}$ is exactly the subgroup of $W$ which preserves
$\Delta^+ \setminus \Delta^+_J$, it suffices to show $w$ preserves
$\Delta^+ \setminus \Delta^+_J$. Let $\alpha \in \Delta^+ \setminus
\Delta^+_J$; then by Proposition \ref{para}, $\lambda - \Z^{\geqslant 0}
\alpha \subset \wt M(\lambda, J)$, whence so is $\lambda - \Z^{\geqslant
0} w(\alpha)$ by Proposition \ref{noholes}. This shows that $w(\alpha) >
0$; it remains to show $w(\alpha) \notin \Delta^+_J$. It suffices to
check this for the simple roots $\alpha_i, i \in I \setminus J$. Suppose
not, i.e.~$w(\alpha_i) \in \Delta^+_J$ for some $i \in I \setminus J$. In
this case, note $w(\alpha_i)$ must be a real root of $\Delta^+_J$, e.g.
by considering $2w(\alpha_i)$. Thus the $w(\alpha_i)$ root string
$\lambda - \Z^{\geqslant 0} w(\alpha_i)$ is the set of weights of an
integrable representation of $\g_{-w(\alpha_i) } \oplus [
\g_{-w(\alpha_i)}, \g_{w(\alpha_i)}] \oplus \g_{w(\alpha_i) } \simeq
\lie{sl}_2$, which is absurd.
\end{proof}

\begin{cor}
The stabilizer of $\ch V$ in $W$ is $W_{I_V}$.
\end{cor}

\begin{cor}\label{Cpoly}
For any $V$, $\conv V$ is the Minkowski sum of $\conv W_{I_V}(\lambda)$
and the cone $\R^{\geqslant 0} W_{I_V}(\pi_{I \setminus I_V})$.
\end{cor}

\begin{proof}
By Theorem \ref{intgrp}, we reduce to the case of $V$ a parabolic Verma
module. Now the result follows from the Ray Decomposition \ref{ray} and
Equation \eqref{Everma}.
\end{proof}

\begin{re}
For $\g$ semisimple, the main theorem \ref{intgrp} and the Ray
Decomposition \ref{ray} imply that the convex hull of weights of any
highest weight module $V$ is a $W_{I_V}$-invariant polyhedron. To our
knowledge, this was known for many but not all highest weight modules by
recent work \cite{Kh1}, and prior to that only for parabolic Verma
modules. We develop many other applications of the main theorem to the
convex geometry of $\conv V$, including the classification of its faces
and their inclusions, in the companion work \cite{DK2}.
\end{re}

For completeness, we record here the description of the extremal rays of
the convex shape $\conv V$, which is only implicit in the work
\cite{DK2}. 

\begin{cor} The extremal rays in $\conv V$ are the rays 
\begin{equation} \label{e:extrays}
       w( \lambda - \mathbb{R}^{\geqslant 0} \alpha_i), \quad \text{for all } w \in W_{I_V} \text{ and }i \in I \setminus I_V. 
\end{equation}
%
\end{cor}

See also the recent preprint \cite{Teja}, where G.K.~Teja employs a
different argument to prove the same result.

\begin{proof} By Theorem~\ref{intgrp}, we may assume $V$
is a parabolic Verma module. We will first show that any extremal ray $R$ of $\conv V$ must lie in the collection \eqref{e:extrays}. To see this, by the Ray Decomposition \ref{ray}, we may write an interior point $r$ of $R$ as a convex combination of points $r_{w, i}$ of the rays in \eqref{e:extrays}. If any $r_{w, i}$ is an interior point of the corresponding ray of \eqref{e:extrays}, we are done. If not, then $R$ would be an extremal ray of $\conv W_{I_V} (\lambda)$. However, as each $$w \lambda,\quad \quad  \text{ for } w \in W_{I_V},$$is a face of $\conv W_{I_V}(\lambda)$, this straightforwardly implies that $\conv W_{I_V}(\lambda)$  contains no extremal rays.

It remains to check that each element of \eqref{e:extrays} is extremal, i.e. is a face of $\conv V$. As $W_{I_V}$ acts by automorphisms of $\conv V$, we may assume $w = e$. For a fixed $i \in I \setminus I_V$, it suffices to show that the ray $$\lambda - \mathbb{R}^{\geqslant 0} \alpha_i$$maximizes the restriction to $\conv V$ of a real functional $\xi: \h^* \rightarrow \mathbb{R}$. To see this, one may pick any $\xi$ satisfying 
\[
 \xi(\alpha_i) =  0 \quad \text{and} \quad \xi(\alpha_j) > 0, \quad \text{for all $j \in I \setminus i$}, 
\]
which completes the proof.
\end{proof}

\section{Three positive formulas for the weights of simple
modules}\label{Sbump}

We remind the three positive formulas for the weights
of simple highest weight modules to be obtained in this section. 

\begin{pro}\label{pp1}
Write $\fl$ for the Levi subalgebra corresponding to $I_\lambda$,
and write $L_\fl(\nu)$ for the simple $\fl$ module of highest weight $\nu
\in \h^*$. Then:
\begin{equation}
\wt L(\lambda) = \bigsqcup_{\mu \in \Z^{\geqslant 0} \pi \setminus
\pi_{I_\lambda} } \wt L_{\mathfrak{l}}(\lambda - \mu).
\end{equation}
\end{pro}

\begin{pro}\label{pp3}
\begin{equation}
\wt L(\lambda) = \conv L(\lambda) \cap \{\mu \in \h^*: \mu \leqslant
\lambda \}.
\end{equation} 
\end{pro}

\begin{pro}\label{pp2}
Suppose $\lambda$ has finite stabilizer in $W_{I_\lambda}$. Then:
\begin{equation}
\wt L(\lambda) = W_{I_\lambda} \{ \mu \in P^+_{I_\lambda}: \mu
\leqslant \lambda \}.
\end{equation}
\end{pro}

\begin{re} For an extension of Proposition \ref{pp2} to arbitrary
$\lambda$, see the companion work \cite{DK2}. \end{re}

Note that the three weight formulas above are immediate consequences of
combining the following theorem with Propositions \ref{slice} and
\ref{noholes} and Corollary \ref{Cfinite} respectively. Recall from
Section \ref{Sbump-results} that for a highest weight module $V$, we
defined its {\em potential integrability} to be $I^p_V := I_\lambda
\setminus I_V$. 

\begin{theo}\label{weights}
Let $V$ be a highest weight module, $V_\lambda$ its highest weight line.
Let $\fl$ denote the Levi subalgebra corresponding to $I^p_V$. Then $\wt
V = \wt M(\lambda, I_V)$ if and only if
\[
\wt U(\mathfrak{l}) V_\lambda = \wt U(\fl) M(\lambda, I_V)_\lambda,
\]
i.e.
$\wt U(\fl) V_\lambda = \lambda - \Z^{\geqslant 0} \pi_{I^p_V}$. In
particular, if $V = L(\lambda)$, then $\wt L(\lambda) = \wt M(\lambda,
I_\lambda)$.
\end{theo}

\begin{proof}
It is clear, e.g.~by the Integrable Slice Decomposition \ref{slice} that
if $\wt V = \wt M(\lambda, I_V)$, then $\wt_{I^p_V} V = \lambda -
\Z^{\geqslant 0} \pi_{I^p_V}$. For the reverse implication, by again
using the Integrable Slice Decomposition and the action of $\fl_{I_V}$,
it suffices to show that $\lambda - \Z^{\geqslant 0} \pi_{I \setminus
I_V} \subset \wt V$.

Suppose not, i.e.~there exists $\mu \in \Z^{\geqslant 0} \pi_{I \setminus
I_V}$ with $V_{\lambda - \mu} = 0$. We decompose $I \setminus I_V = I^p_V
\sqcup I^q$, and write accordingly $\mu = \mu_p + \mu_q$, where $\mu_p
\in \Z^{\geqslant 0} \pi_{I^p_V}$ and $\mu_q \in \Z^{\geqslant 0}
\pi_{I^q}$.
By assumption $V_{\lambda - \mu_p}$ is nonzero, so by the PBW theorem,
there exists $\partial \in U(\n^-_{I^p_V})$ of weight $-\mu_p$ such that
$\partial \cdot V_{\lambda}$ is nonzero. Now choose an enumeration of
$I^q = \{i_k : 1 \leqslant k \leqslant n \},$ write $\mu_q = \sum_k m_k
\alpha_{i_k}$, and consider the monomials $E := \prod_k e_{i_k}^{m_k}, F
:= \prod_k f_{i_k}^{m_k}$. By the vanishing of $V_{\lambda - \mu_p -
\mu_q}, \partial F$ annihilates the highest weight line $V_{\lambda}$,
whence so does $E \partial F$.
Since $[e_i, f_j] = 0\ \forall i \neq j$, we may write this as $\partial
\prod_k e_{i_k}^{m_k} f_{i_k}^{m_k}$, and each factor $e_{i_k}^{m_k}
f_{i_k}^{m_k}$ acts on $V_{\lambda}$ by a nonzero scalar, a
contradiction.
\end{proof}

\begin{re}
Given the delicacy of Jantzen's criterion for the simplicity of a
parabolic Verma module (see \cite{Jantzen,Jantzen-thesis}), it is
interesting that the equality on weights of $L(\lambda)$ and $M(\lambda,
I_\lambda)$ always holds.
\end{re}

\begin{re}
{Since the writing of this paper, G.K.~Teja has obtained
in~\cite{Teja}} a `minimal description' of the set of weights of
$L(\lambda)$ for arbitrary non-integrable $\lambda$ -- and more
generally, for any parabolic Verma module as in Equation~\eqref{Everma}.
{Explicitly, Teja shows that
\[
\wt M(\lambda, J) = \wt L_{\hspace{.5mm} \fl_J}(\lambda) -
\Z^{\geqslant 0} \Delta_{J^c, 1},
\]
where given a subset $J \subset I$ of simple roots, the set
$\Delta_{J^c,1}$ comprises the positive roots $\alpha_{i_1} + \cdots +
\alpha_{i_n}$ with ``$J^c$-height $1$'', i.e., with a unique $k$ such
that $i_k \not\in J$. (In particular, for $J = I_\lambda$ we obtain $\wt
L(\lambda)$.) The `minimality' of this description is the assertion that
the set $\Delta_{J^c, 1}$ in the above formula cannot be reduced.}
\end{re}

\section{Weights and convex hulls of highest weight
modules}\label{Slepowsky}

In Section \ref{Sbump}, we applied our main result to obtain formulas for
the weights of simple modules. We now explore the extent to which this
can be done for arbitrary highest weight modules.

\begin{theo}\label{eek}
Fix $\lambda \in \h^*$ and $J \subset I_\lambda$. Every highest
weight module $V$ of highest weight $\lambda$ and integrability $J$ has
the same weights if and only if $J^p := I_\lambda \setminus J$
is complete, i.e. $(\halpha_j, \alpha_{j'}) \neq 0, \forall j, j' \in
J^p$. 
\end{theo}


We first sketch our approach. The subcategory of modules $V$ with fixed
highest weight $\lambda$ and integrability $J$ is basically a poset
modulo scaling. The source here is the parabolic Verma module
$M(\lambda,J)$, and we are concerned with the possible vanishing of
weight spaces as we move away from the source. The most interesting
ingredient of the proof is the observation that there is a sink, which we
call $L(\lambda, J)$, and hence the question reduces to whether
$M(\lambda,J)$ and $L(\lambda,J)$ have the same weights.

\begin{lem}\label{Lkitty}
For $J \subset I_\lambda$, there is a minimal quotient $L(\lambda,
J)$ of $M(\lambda)$ satisfying the equivalent conditions of Theorem
\ref{intgrp}.
\end{lem}

\begin{proof}
Consider all submodules $N$ of $M(\lambda)$ such that $N_{\lambda -
\Z^{\geqslant 0} \alpha_i} = 0, \forall i \in I \setminus J$. There is a
maximal such, namely their sum $N'$, and it is clear that $L(\lambda, J)
= M(\lambda)/N'$ by construction. 
\end{proof}

As the reader is no doubt aware, the inexplicit construction of
$L(\lambda,J)$ here is parallel to that of $L(\lambda)$ -- in fact,
$L(\lambda) = L(\lambda, I_\lambda)$. The existence of the objects
$L(\lambda, J)$ was noted by the second named author in \cite{Kh1}, but
we are unaware of other appearances in the literature. In \cite{Kh1}, the
character of $L(\lambda, J)$ was not determined. However, it turns out to
be no more difficult than $\ch L(\lambda)$ under the following sufficient
condition, which as we explain below is often satisfied:

\begin{pro}\label{Pfinite}
Fix $\lambda \in \h^*, J \subset I_\lambda$. Suppose that
$\Ext^1_{\mathscr{O}}(L(\lambda), L(s_i \cdot \lambda)) \neq 0,\
\forall i \in I_\lambda \setminus J$. Then there is a short exact
sequence:
\begin{equation}\label{lj}
0 \rightarrow \bigoplus_{i \in I_\lambda \setminus J}L(s_i \cdot
\lambda) \rightarrow L(\lambda, J) \rightarrow L(\lambda) \rightarrow
0.
\end{equation}
\end{pro}

\begin{proof}
The choice of a nonzero class in each $\Ext^1_{\mathscr{O}}(L(\lambda),
L(s_i \cdot \lambda)), i \in I_\lambda \setminus J$ gives an
extension $E$ as in Equation \eqref{lj}. The space $E_\lambda$ is a
highest weight line, and we obtain an associated map $M(\lambda) \simeq
M(\lambda) \otimes E_\lambda \rightarrow E$. This is surjective, as
follows from an easy argument using Jordan--H\"older content and the
nontriviality of each extension. The consequent surjection $E \rightarrow
L(\lambda, J)$ is an isomorphism, by considering Jordan--H\"older content.
\end{proof}

\begin{re}
By identifying $\Ext^1_{\mathscr{O}}(L(\lambda),
L(s_i \cdot \lambda))$ with $\Hom(N(\lambda), L(s_i
\cdot \lambda))$, where $N(\lambda)$ is the maximal submodule of
$M(\lambda)$, one sees it is at most one dimensional. 

Suppose $\mathfrak{g}$ is symmetrizable, and that $\lambda$ is a (i)
noncritical weight which (ii) lies in the dot Weyl group orbit of a
dominant or antidominant weight. If $\lambda$ and $s_i \cdot \lambda$ are
distinct and lie in the same block, then we expect that the extension
group
\[
\Ext^1_{\mathscr{O}}(L(\lambda), L(s_i \cdot \lambda))
\]
is nonzero. Indeed, by work of Fiebig \cite{Fiebig}, we may assume that
$\lambda$ is integral. Let us sketch a proof in the case where $\lambda$
is regular and lies in the orbit of an antidominant weight. By
Kashiwara--Tanisaki localization \cite{kt}, we must equivalently show the
nonvanishing of the extensions between the corresponding intersection
cohomology sheaves on $B\backslash G/N$. Using parity vanishing, the
Cousin spectral sequence converging to $$\Ext^*_{\mathscr{O}}(L(\lambda),
L(s_i \cdot \lambda))$$ collapses on its first page. This reduces the
nonvanishing assertion to only the Schubert strata indexed by $s_i \cdot
\lambda$ and $\lambda$, where it is clear. 

A similar argument applies if $\lambda$ lies in the orbit of an
antidominant weight and has a finite stabilizer. In particular, this
covers all cases in finite type (where the result is well known, cf.
\cite[Theorem 8.15(c)]{H3} for $\lambda$ regular and see \cite[Corollary
1.3.5]{Irving} for $\lambda$ singular). A similar argument should apply
for orbits containing dominant weights with finite stabilizers, although
the convergence of the spectral sequence requires justification. One can
also see this nonvanishing directly if $\lambda$ is a regular dominant
weight by using the last two terms of the BGG resolution, cf. the proof
of Proposition \ref{dogs} below.
\end{re}

Returning to Theorem \ref{eek} for $\g$ possibly non-symmetrizable,
we first prove the following weaker form of Equation
\eqref{lj}.

\begin{pro}\label{dogs}
Let $\g$ be of arbitrary type and $\lambda \in P^+$. Then:
\begin{equation}\label{dinosaurs}
\wt L(\lambda, J) = \wt L(\lambda) \cup \bigcup_{i \in I \setminus J} \wt
L(s_i \cdot \lambda).
\end{equation}
\end{pro}

\begin{proof}
The restriction of $\ch M(\lambda)$ to the root string $\lambda -
\Z^{\geqslant 0} \alpha_i, i \in I \setminus J$ is entirely determined by
the Jordan--H\"older content $[M(\lambda): L(\lambda)] = [M(\lambda):
L(s_i \cdot \lambda)] = 1$. Using this, it suffices to exhibit a module
with weights given by the right-hand side of \eqref{dinosaurs}.

Consider the end of the BGG resolution:
\begin{equation}\label{bgge}
\bigoplus_{i \in I} M(s_i \cdot \lambda) \rightarrow M(\lambda)
\rightarrow L^{\max}(\lambda) \rightarrow 0.
\end{equation}
For each $i \in I$, let $N(s_i \cdot \lambda)$ denote the maximal
submodule of $M(s_i \cdot \lambda)$.
Consider the quotient $Q$ of $M(\lambda)$ by the image of $\bigoplus_{j
\in J} M(s_j \cdot \lambda) \oplus \bigoplus_{i \in I \setminus J} N(s_i
\cdot \lambda)$, under \eqref{bgge}. We deduce:
\[
\ch Q = \ch L^{\max}(\lambda) + \sum_{i \in I \setminus J} \ch L(s_i
\cdot \lambda),
\]
which implies that $\wt Q$, whence $\wt L(\lambda,J)$, coincides with the
right-hand side of \eqref{dinosaurs}.
\end{proof}

We now are ready to prove Theorem \ref{eek}. By the above results, it
suffices to determine when $\wt L(\lambda, J) = \wt M(\lambda, J)$.

\begin{proof}[Proof of Theorem \ref{eek}]
It will be clarifying, though not strictly necessary for the argument, to
observe the following compatibility of the construction of $L(\lambda,
J)$ and restriction to a Levi.

\begin{lem}\label{compa}
Let $I' \subset I$, $\fl = \fl_{I'}$ the corresponding Levi subalgebra,
and for $J \subset I$ write $J' = J \cap I'$. Then:
\[
L_{\hspace{.5mm} \fl}(\lambda, J') \simeq U(\fl) L(\lambda, J)_{\lambda},
\]
where $L_{\hspace{.5mm} \fl}(\lambda,J')$ is the smallest $\fl$-module
with integrability $J'$, constructed as in Lemma \ref{Lkitty}.
\end{lem}

\begin{proof}
Since the integrability of $U(\fl) L(\lambda, J)_{\lambda}$ is $J'$, we
have a surjection:
\begin{equation}\label{kitty}
U(\fl) L(\lambda, J)_{\lambda} \rightarrow L_{\hspace{.5mm} \fl}(\lambda,
J') \rightarrow 0.
\end{equation}

\noindent To see that Equation \eqref{kitty} is an isomorphism, we need
to show the kernel $K$ of $M_{\fl}(\lambda, J') \rightarrow
L_{\hspace{.5mm} \fl}(\lambda, J')$ is sent into the kernel of
$M(\lambda, J) \rightarrow L(\lambda, J)$ under the inclusion
$M_{\fl}(\lambda, J') \rightarrow M(\lambda, J)$. To do so, choose a
filtration of $K$:
\[
0 \subseteq K_1 \subseteq K_2 \subseteq \cdots, \qquad K = \cup_{m
\geqslant 0} K_m,
\]

\noindent where $K_m/K_{m-1}$ is generated as a $U(\fl)$-module by a
highest weight vector $v_m$. 

It suffices to prove by induction on $m$ that the $U(\g)$-module
generated by $K_m$ lies in the kernel of $M(\lambda, J) \rightarrow
L(\lambda, J)$, $\forall m \geqslant 0$. For the inductive step, suppose
that $M(\lambda, J)/U(\g) K_{m-1}$ has integrability $J$. Then choosing a
lift $\tilde{v}_m$ of $v_m$ to $K_m$, observe that $\tilde{v}_m$ is a
highest weight vector in $M(\lambda, J)/U(\g) K_{m-1}$ for the action of
all of $U(\g)$, by the linear independence of simple roots. It is clear
that $U(\g) \tilde{v}_m$ has trivial intersection with the $(\lambda -
\Z^{\geqslant 0} \alpha_i)$-weight spaces of $M(\lambda, J)/U(\g)
K_{m-1}$, for all $i \in I \setminus J$, hence $M(\lambda, J)/U(\g) K_m$
again has integrability $J$.
\end{proof}

We are now ready to prove Theorem \ref{eek}. By Theorem \ref{weights}, we
have $\wt L(\lambda, J) = \wt M(\lambda, J)$ if and only if $\wt_{J^p}
L(\lambda, J) = \wt_{J^p} M(\lambda, J)$. By Lemma \ref{compa}, we may
therefore assume that $J = \emptyset$, and $\lambda \in P^+$, and hence
$J^p = I$.

By Proposition \ref{dogs}, we have:
\begin{align*}
\wt L(\lambda, \emptyset) = &\ \wt L(\lambda) \cup \bigcup_{i \in I} \wt
L(s_i \cdot \lambda),\\
\wt L(\lambda, I \setminus i) = &\ \wt L(\lambda) \cup \wt L(s_i \cdot
\lambda), \quad \forall i \in I.
\end{align*}

\noindent Combining the two above equalities, we have:
\[
\wt L(\lambda, \emptyset) = \bigcup_{i \in I} \wt L(\lambda, I \setminus
i).
\]

\noindent For $L(\lambda, I \setminus i)$ the set of potentially
integrable directions is $\{i\}$, so we may apply Theorem \ref{weights}
to conclude:
\begin{equation}\label{pvf}
\wt L(\lambda, \emptyset) = \bigcup_{i \in I} \wt M(\lambda, I \setminus
i).
\end{equation}
Thus, it remains to show that for $\lambda$ dominant integral,
$\wt M(\lambda) = \bigcup_{i \in I} \wt M(\lambda, I \setminus
i)$ if and only if the Dynkin diagram of $\lie{g}$ is complete.

Suppose first that the Dynkin diagram of $\g$ is not complete.
Pick $i,i' \in I$ with $(\halpha_i, \alpha_{i'}) = 0$. Then $M(\lambda)$
and $M(\lambda)/ s_i s_{i'} \cdot M(\lambda)_\lambda$ have distinct
weights, as can be seen by a calculation in type $A_1 \times A_1$. 

Suppose now that the Dynkin diagram of $\g$ is complete. By \eqref{pvf}
it suffices to show:
\[
\wt M(\lambda, I \setminus i) \supset \{ \lambda - \sum_{j \in I} k_j
\alpha_j: k_j \in \Z^{\geqslant 0}, k_i \geqslant k_j, \forall j \in I
\}.
\]

\noindent But this follows using the assumption that every node is
connected to $i$, and the representation theory of $\lie{sl}_2$.
\end{proof}

\begin{re}
Theorem \ref{eek} can be proved without working directly with the objects
$L(\lambda,J)$, but instead using only the lower bound \eqref{dinosaurs}
for the weights:
\[
\wt V \supset \wt L(\lambda, I_V) = \wt L(\lambda) \cup \bigcup_{i \in
I_\lambda \setminus I_V} \wt L(s_i \cdot \lambda).
\]
\end{re}

\section{An alternating formula for the weights of simple
modules}\label{Sweylkac}


The main goal of this section is to prove the following:

\begin{theo}\label{Tsimple}
Let $\lambda \in \h^*$ be such that the stabilizer of $\lambda$ in
$W_{I_\lambda}$ is finite. Then:
\begin{equation}\label{ewks}
\wt L(\lambda) = \sum_{w \in W_{I_\lambda} } w
\frac{e^{\lambda}}{\prod_{\alpha \in \pi} (1 - e^{-\alpha})}.
\end{equation}
\end{theo}

We remind that on the left hand side of \eqref{ewks} we mean the
`multiplicity-free' character $\sum_{\mu \in \h^*: L(\lambda)_\mu \neq 0}
e^\mu$. On the right hand side of \eqref{ewks}, in each summand we take
the `highest weight' expansion of $w(1 - e^{-\alpha_i})^{-1}$, i.e.:
\begin{equation} \label{hwex}
w \frac{1}{1 - e^{-\alpha_i}} := \begin{cases}
1 + e^{-w \alpha_i} + e^{-2 w \alpha_i} + \cdots, & w \alpha_i > 0, \\
- e^{ w\alpha_i} - e^{2w \alpha_i} - e^{3w\alpha_i} - \cdots, & w
\alpha_i < 0.
\end{cases}
\end{equation}

By Theorem \ref{weights}, we will deduce Theorem \ref{Tsimple} from the
following:

\begin{pro}\label{wkpv}
Fix $\lambda \in \h^*$, $J \subset I_\lambda$ such that the
stabilizer of $\lambda$ in $W_J$ is finite. Then:
\[
\wt M(\lambda, J) = \sum_{w \in W_J} w \frac{e^{\lambda}}{\prod_{\alpha
\in \pi} (1 - e^{-\alpha})}.
\]
\end{pro}

Before proving Proposition \ref{wkpv}, we note the similarity between the
result and the following formulation of the Weyl--Kac character formula.
This is due to Atiyah and Bott \cite{atbo} for finite dimensional simple
modules in finite type, but appears to be new in this generality:

\begin{pro} \label{abott}
Fix $\lambda \in \h^*$, $J \subset I_\lambda$. Then:
\begin{equation}\label{abf}
\ch M(\lambda, J)  = \sum_{w \in W_J} w \frac{e^\lambda}{\prod_{\alpha >
0} (1 - e^{-\alpha})^{\dim \g_\alpha}}.
\end{equation}
\end{pro}

The right hand side of Equation \eqref{abf} is expanded in the manner
explained above. We now prove Proposition \ref{wkpv}, and then turn to
Proposition \ref{abott}. 

\begin{proof}[Proof of Proposition \ref{wkpv}]
By expanding $w  \prod_{\alpha \in \pi_{I \setminus J}}(1 -
e^{-\alpha})^{-1}$, note the right hand side equals: $$\sum_{w \in W_J}
\sum_{\mu \in \Z^{\geqslant 0} \pi_{I \setminus J}} w \frac{e^{\lambda -
\mu}}{\prod_{\alpha \in \pi_J} (1 - e^{-\alpha}) }.$$Therefore we are
done by the Integrable Slice Decomposition \ref{slice} and the existing
result for integrable highest weight modules with finite stabilizer
\cite{Kass}. 
\end{proof}

Specialized to the trivial module in finite type, we obtain the following
curious consequence, which roughly looks like a Weyl denominator formula
without a choice of positive roots:

\begin{cor}
Let $\g$ be of finite type with root system $\Delta$, and let $\Pi$
denote the set of all  $\pi$, where $\pi$ is a simple system of roots for
$\Delta$. Then we have:
\begin{equation}
\prod_{\alpha \in \Delta} (1 - e^{\alpha}) = \sum_{\pi \in \Pi}
\prod_{\beta \notin \pi} (1 - e^\beta).
\end{equation}
\end{cor}

\begin{re}\label{otherone}
When $\g$ is of finite type, Theorem \ref{Tsimple} follows from
Proposition \ref{p3} and general formulae for exponential sums over
polyhedra due to Brion \cite{Brion} for rational polytopes and Lawrence
\cite{Lawrence} and Khovanskii--Pukhlikov \cite{KP} in general,
cf.\ \cite[Chapter 13]{Barvinok-points}. For regular weights, one uses
that the tangent cones are unimodular and hence the associated polyhedra
are Delzant. For general highest weights one may apply a deformation
argument due to Postnikov \cite{Postnikov}. We thank Michel Brion for
sharing this observation with us.

Moreover, it is likely that a similar approach works in infinite type,
though a cutoff argument needs to be made owing to the fact that $\conv
V$ is in general locally, but not globally, polyhedral, cf.\ our
companion work \cite{DK2}. 
\end{re}

\subsection{Some related results}

We conclude this section with several related observations which are not
needed in the remainder of the paper, beginning with Proposition
\ref{abott}.

\subsubsection{Parabolic Atiyah--Bott formula and the
Bernstein--Gelfand--Gelfand--Lepowsky resolution}

Proposition \ref{abott} may be deduced from the usual presentation of the
parabolic Weyl--Kac character formula:

\begin{pro}\label{squeak}
Fix $\lambda \in \h^*, J \subset I_\lambda$. Then:
\begin{equation}
\ch M(\lambda, J) = \sum_{w \in W_J} \frac{ e^{w(\lambda + \rho) -
\rho}}{\prod_{\alpha> 0} (1 - e^{-\alpha})^{\dim
\g_\alpha}}.
\end{equation}
\end{pro}

To our knowledge, even Proposition \ref{squeak} does not appear in the
literature in this generality. However, via Levi induction one may reduce
to the case of $\lambda$ dominant integral, $J = I$, where it is a famed
result of Kumar \cite{kumar-demazure,kbg}. Along these lines, we also
record the BGGL resolution for arbitrary parabolic Verma modules, which
may be of independent interest:

\begin{pro}\label{ggg}
Fix $M(\lambda, J)$ and $J' \subset J$. Then $M(\lambda, J)$ admits a
BGGL resolution, i.e., there is an exact sequence:
\begin{align}
\cdots\ \rightarrow \bigoplus_{w \in W^{J'}_J: \ell(w) = 2} M(w \cdot
\lambda, J') \rightarrow \bigoplus_{w \in W^{J'}_J: \ell(w) = 1} M(w
\cdot \lambda, J') &\ \rightarrow\\
\rightarrow M(\lambda, J') \rightarrow &\ M(\lambda, J) \rightarrow
0.\notag
\end{align}

\noindent Here $W^{J'}_J$ runs over the minimal length coset
representatives of $W_{J'}\backslash W_J$. In particular, we have the
Weyl--Kac character formula:
\begin{equation}\label{Echar2}
\ch M(\lambda, J) = \sum_{w \in W^{J'}_J} (-1)^{\ell(w) } \ch M(w \cdot
\lambda, J').
\end{equation}
\end{pro}

This can be deduced by Levi induction from the result of Heckenberger and
Kolb \cite{heko}, which builds on \cite{lepo2,kbg,murr}.

\subsubsection{Generalization of a result of Kass}

In the paper of Kass \cite{Kass} that proves Theorem \ref{Tsimple} for
integrable modules, the main result is a recursive formula for the
character of an integrable module. After establishing notation, we will
state the analogous result for parabolic Verma modules. 

For $J \subset I$, if $\nu \in \h^*$ lies in the $J$ dot Tits cone, write
$\overline{\nu} = w_\nu \cdot \nu$ for the unique $J$ dot dominant weight
in its $W_J$ orbit. For $\lambda \in P^+_J$, we write $\chi_{\lambda} :=
\ch M(\lambda, J)$, and for $\nu$ lying in the dot $J$ Tits cone, we set:
\[
\chi_{\nu} := \begin{cases}  (-1)^{\ell(w_\nu)}
\chi_{\overline{\nu}} & \text{if $\overline{\nu}$ is in the $J$ dominant
chamber,} \\ 0 & \text{otherwise.} \end{cases}
\]

\noindent This makes sense because if $\overline{\nu}$ lies in the $J$
dominant chamber, it is regular dominant under the $J$ dot action, so
there is a unique $w_\nu \in W_J$ such that $w_\nu \cdot \nu =
\overline{\nu}$.

Define the elements $\langle \Phi \rangle \in \Z^{\geqslant 0} \pi$ and
the integers $c_{\langle \Phi \rangle}$ via the expansion,
\[
\prod_{\alpha \in \Delta^+ \setminus \pi} (1 - e^{-\alpha})^{\dim
\g_\alpha} =: \sum_{\langle \Phi \rangle} c_{\langle \Phi \rangle} e^{ -
\langle \Phi \rangle}.
\]

\noindent Informally, this expansion counts the number of ways to write
$\langle \Phi \rangle$ as sums of non-simple positive roots, with
consideration of the multiplicity of the root and the parity of the sum.

For a parabolic Verma module $M(\lambda, J)$, by abuse of notation
identify its weights with the corresponding multiplicity-free character:
\[
\wt M(\lambda, J) = \sum_{\mu: M(\lambda, J)_{\mu} \neq 0} e^{\mu}.
\]

With this notation, we are ready to state the following:

\begin{pro}\label{sc}
Let $M(\lambda, J)$ be a parabolic Verma module such that the stabilizer
of $\lambda$ in $W_J$ is finite. Then:

\begin{equation}\label{stwo}
\wt M(\lambda, J) = \sum_{\langle \Phi \rangle} c_{\langle \Phi \rangle}
\chi_{ \lambda - \langle \Phi \rangle },
\end{equation}
Furthermore, for all $\langle \Phi \rangle$, we have
(i) $\overline{\lambda - \langle \Phi \rangle} \leqslant \lambda$,
(ii) with equality if and only if $\Phi$ is empty.
\end{pro}

In Equation \eqref{stwo}, we expand denominators as `highest weight'
power series as explained in Equation \eqref{hwex}. The arguments of
\cite{Kass} or \cite{Sc} apply with suitable modification to prove
Proposition \ref{sc}.

\begin{re}
As the remainder of the paper concerns only symmetrizable $\g$, we now
explain the validity of earlier results for $\overline{\g}$ (cf.\
Section~\ref{rat}), should $\g$ and $\overline{\g}$ differ. The proofs in
Sections \ref{S4}--\ref{Sweylkac} prior to Remark
\ref{otherone} apply verbatim for $\overline{\g}$. Alternatively,
recalling that the roots of $\g$ and $\overline{\g}$ coincide \cite[\S
5.12]{Kac}, it follows from Equation~\eqref{Everma} and the remarks
following Proposition~\ref{nonde} that the weights of parabolic Verma
modules for $\g, \overline{\g}$ coincide, as do the weights of simple
highest weight modules. Hence many of the previous results can be deduced
directly for $\overline{\g}$ from the case of $\g$. These arguments for
$\overline{\g}$ apply for any intermediate Lie algebra between $\g$ and
$\overline{\g}$.
\end{re}

\section{A
geometric interpretation of integrability}\label{Sgeometry}

Throughout this section, $\g$ is of finite type. We will answer the
question of Brion \cite{Brion2} discussed in Section \ref{Sbrion2}. We
first remind notation. For $\lambda$ a dominant integral weight, consider
$\mathscr{L}_\lambda$, the line bundle on $G/B$ with
$H^0(\mathscr{L}_\lambda) \simeq L(\lambda)$ as $\g$-modules. For $w \in
W$, write $C_w := BwB/B$ for the Schubert cell, and write $X_w$ for its
closure, the Schubert variety. To affirmatively answer Brion's question
for any regular integral infinitesimal character, it suffices by
translation to consider the regular block $\mathscr{O}_0$.

Let us recall the geometric side of the localization theorem. For an
Artin stack $X$, we write $D(X)$ for its bounded derived category of
D-modules with coherent cohomology. We write $N$ for the unipotent
radical of $B$, and consider the associated category of D-modules on the
quotient stack $B \backslash G / N$, i.e. 
\begin{equation} \label{e:dmod}
     D(B \backslash G / N). 
\end{equation}
As every object of \eqref{e:dmod} has regular holonomic cohomology, the
Riemann--Hilbert correspondence gives an equivalence with the bounded
constructible derived category of sheaves 
\begin{equation} \label{e:rh}
    D(B \backslash G / N) \simeq \operatorname{Sh}(B \backslash G / N),
\end{equation}
which exchanges the standard $t$-structure on D-modules with the perverse
$t$-structure on constructible sheaves. 

Let us next recall the form of Beilinson--Bernstein localization we will
use. Consider the action of the abstract Cartan $T \simeq B/N$ on $B
\backslash G / N$ by right translation. Note that all the objects of
\eqref{e:dmod} are $T$-monodromic, i.e. lie in the full subcategory
generated under colimits by $!$-pulling back along the projection 
\[
   B \backslash G / N \rightarrow (B \backslash G / N) / T \simeq B
   \backslash G / B.
\]
In particular, they carry a canonical datum of {\em weak}
$T$-equivariance. Accordingly, one may form the $T$-invariant sections of
the underlying quasicoherent sheaf of the (left) D-module on $G/N$, i.e.
one has a functor
\[
\Gamma(G/B, -): D(B \backslash G / N) \rightarrow  \operatorname{Vect},
\]
where $\operatorname{Vect}$ denotes the bounded derived category of
vector spaces. The canonical embedding of $\g$ into global vector fields
on $G/N$ factors $\Gamma(G/B, -)$ as a composition 
\[
     D(B \backslash G / N) \rightarrow \g\operatorname{-mod}
     \xrightarrow{\operatorname{Oblv}} \operatorname{Vect},
     \]
where $\g\operatorname{-mod}$ denotes the bounded derived category of
$\g$-modules and $\operatorname{Oblv}$ denotes the forgetful functor.
Then the theorem of Beilinson--Bernstein \cite[Corollary 3.3.3]{bbj}
asserts that the obtained functor to $\g$-modules induces a $t$-exact
equivalence with the bounded derived category of $\mathscr{O}_0$, i.e. 
\begin{equation} \label{e:bbloc}
     \Gamma(G/B, -): D(B \backslash G / N) \simeq D^b(\mathscr{O}_0).
\end{equation}

Before turning the problem at hand, let us briefly comment on the
interaction of the above equivalences with duality. For \eqref{e:rh}, we
recall that the Riemann--Hilbert correspondence exchanges the standard
duality on D-modules with Verdier duality for constructible sheaves, see
for example \cite[Corollary 4.6.5]{htt}. For \eqref{e:bbloc}, it is known
that localization exchanges the standard duality on D-modules with the
contragredient duality on $D^b(\mathscr{O})$, see for example Section 1
of \cite{arkhgaits}. In particular, we use the normalization of
\cite{arkhgaits} for the contragredient duality. We refer to all three
dualities by $\mathbb{D}$ below. 

Having established our conventions, we turn to Brion's question. Let $V$
be a highest weight module with highest weight $w \cdot (-2\rho), w \in
W$. Write $\mathscr{V}$ for the corresponding regular holonomic
$D$-module on $G/B$, and $\mathscr{V}_{rh}$ for the corresponding
perverse sheaf under the Riemann--Hilbert correspondence. Then
$\mathscr{V}, \mathscr{V}_{rh}$ are supported on $X_w$. Note that for $i
\in I_V$, $X_{s_iw}$ is a Schubert divisor of $X_{w}$. 

\begin{theo}\label{brg}
Let $V$ and $\mathscr{V}$ be as above. For a subset $J \subset I_V$,
consider the corresponding union of Schubert divisors
\[
Z := \bigcup_{i \in I_V \setminus J} X_{s_iw},
\]
and denote its complement by $\tu := (G/B) \setminus Z.$ Then the
$\g$-module $\mathbb{D} H^0(\tu, \mathbb{D} \mathscr{V})$ is of highest
weight $w \cdot (-2\rho)$ and has integrability $J$.
\end{theo}

To prove Theorem \ref{brg}, we will use the following geometric
characterization of integrability, which was promised in Theorem
\ref{maintheo}.

\begin{pro}\label{poles}
For $J \subset I_\lambda$, consider the smooth open subvariety of
$X_w$ given on complex points by:
\begin{equation}
\cU_J := C_w \sqcup \bigsqcup_{i \in J} C_{s_iw}.
\end{equation}

\noindent Let $V, \cV_{rh}$ be as above, and set $\cU :=
\cU_{I_\lambda}$. Upon restricting to $\cU$ we have:
\begin{equation}
\cV_{rh}|_{\cU} \simeq j_! \C_{\cU_{I_V}} [\ell(w)].
\end{equation}

\noindent In particular, $I_V = \{i \in I: \text{the $*$-stalks of
$\cV_{rh}$ along $C_{s_iw}$ are nonzero}\}.$
\end{pro}

\begin{proof}
For $J \subset I_\lambda$, let $P_J$ denote the corresponding
parabolic subgroup of $G$, i.e. with Lie algebra $ \fl_J + \n^+$. Then it
is well known that the perverse sheaf corresponding to $M(\lambda, J)$ is
$j_! \C_{P_JwB/B} [\ell(w)]$. 

The map $M(\lambda, I_V) \rightarrow V \rightarrow 0$ yields a surjection
on the corresponding perverse sheaves. By considering Jordan--H\"{o}lder
content, it follows this map is an isomorphism when restricted to
$\cU_{I_\lambda}$, as for $y \leqslant w$ the only intersection
cohomology sheaves $\IC_y := \IC_{X_y}$ which do not vanish upon
restriction are $\IC_w, \IC_{s_iw}, i \in I_\lambda.$ We finish by
observing:
\[
j_{\cU_{I_\lambda} }^* j_! \C_{\cU_{I_V}}[\ell(w)] \simeq j_!
\C_{\cU_{I_\lambda} \cap P_{I_V}wB/B}[\ell(w)] = j_! \C_{\cU_{I_V}}
[\ell(w)].\qedhere
\]
\end{proof}

We deduce the following $D$-module interpretation of integrability:

\begin{cor}\label{Cqbrion}
Let $\cV, \cU$ be as above. The restriction of $\mathbb{D} \cV$ to $\cU$
is $j_* \OO_{\cU_{I_V}}$. If we define $I_{\cV} := \{i \in I: \text{the
fibers of $\mathbb{D} \cV$ are nonzero along $C_{s_i w}$}\}$, then
$I_{\cV} = I_V$. Here we mean fibers in the sense of the underlying
quasi-coherent sheaf of $\mathbb{D} \cV$. 
\end{cor}

Before proving Theorem \ref{brg}, we first informally explain the idea.
Proposition \ref{poles} says that for $i \in I_{L(w \cdot - 2\rho)}$, the
action of the corresponding $\lie{sl}_2 \rightarrow \g$ on $V$ is not
integrable if and only if $\mathscr{V}_{rh}$ has a `pole' on the Schubert
divisor $X_{s_i w}$. Therefore to modify $V$ so that it loses
integrability along $X_{s_i w}$, we will  restrict $\cV_{rh}$ to the
complement and then extend by zero.

\begin{proof}[Proof of Theorem \ref{brg}]
For ease of notation, write $X := X_w$. Write $X = Z \sqcup U$, where $Z$
is as above, and $G/B = Z \sqcup \tu$. Note  $\tu \cap X_w = U$. Write
$j_U: U \rightarrow X_w, j_{\tu}: \tu \rightarrow G/B$ for the open
embeddings, and $i_Z: Z \rightarrow X_w, i'_{Z}: Z \rightarrow G/B,
i_{X_w}: X_w \rightarrow G/B$ for the closed embeddings. 

We will study $H^0(\tu, \mathscr{V})$ by studying the behavior of
$\mathscr{V}$ on $U$, i.e. before $*$-extending from $X_w$. Formally,
write $\mathscr{V}_{rh} = {i_X}_* {i_X}^* \mathscr{V}_{rh} =: {i_X}_*
\mathscr{P}$, where $\mathscr{P}$ is perverse. The distinguished
triangle:
\begin{equation}\label{ples}
{j_U}_! {j_U}^* \rightarrow \text{id} \rightarrow {i_Z}_* {i_Z}^*
\xrightarrow{+1}
\end{equation}

\noindent gives the following exact sequence in perverse cohomology:
\[
0 \rightarrow \pH^{-1} {i_Z}_* {i_Z}^* \cP \rightarrow \pH^{0} {j_U}_!
{j_U}^* \cP \rightarrow  \cP \rightarrow \pH^0 {i_Z}_* {i_Z}^* \cP
\rightarrow 0.
\]

\noindent We now show in several steps that $\pH^0 {j_U}_! {j_U}^* \cP$
is a highest weight module with the desired integrability.\medskip

\noindent \textit{Step 1: $\pH^0 {j_U}_! {j_U}^* \cP$ is a highest weight
module of highest weight $w \cdot - 2\rho$.}

By definition, we have a surjection $ {j_{C_w}}_!
\C_{C_w}[\ell(w)] \rightarrow \cP \rightarrow 0$. Since ${j_{C_w}}_!
\C_{C_w}$ is supported off of the Schubert divisors, we have
${j_U}_! {j_U}^* {j_{C_w}}_! \C_{C_w} \simeq {j_{C_w}}_! \C_{C_w}$.
By right exactness, we obtain
${j_{C_w}}_! \C_{C_w}[\ell(w)] \rightarrow \pH^0 {j_U}_! {j_U}^* \cP
\rightarrow 0,$ as desired. We remark that similarly, $\pH^0 {i_Z}_*
{i_Z}^* \cP = 0$.\medskip

\noindent \textit{Step 2: The integrability of $\pH^0 {j_U}_!
{j_U}^* \cP$ is $J$.}
In `ground to earth' terms, by Proposition \ref{poles} we need to look at
this sheaf on $\cU$, where by design it has the correct behavior. More
carefully, we have:
\begin{align*}
j_{\cU}^* \pH^0 {j_U}_! {j_U}^* \cP & \simeq \pH^0 j_{\cU}^* {j_U}_!
{j_U}^* \cP \\
& \simeq \pH^0 {j_{\cU \cap U}}_! j_{\cU \cap U}^* \cP  \\
& \simeq \pH^0 {j_{\cU \cap U}}_! j_{\cU \cap U}^* {j_{\cU_{I_V}} }_!
\C[\ell(w)] \\
& \simeq \pH^0 {j_{\cU_{I_V} \cap U}}_! \C[\ell(w)] =  {j_{\cU_{I_V} \cap
U}}_! \C[\ell(w)].
\end{align*}

\noindent As $\cU_{I_V} \cap U = \cU_J$, we are done by Proposition
\ref{poles}. 

We now push our analysis off of $X_w$. To do so, we use the standard
isomorphism of distinguished triangles:
\begin{equation}\label{ddt}
\begin{CD}
{j_{\tu}}_! {j_{\tu}}^* {i_X}_* @>>> {i_X}_* @>>> {i'_{Z}}_! {i'_{Z}}^*
{i_X}_* @>+1>>\\
@VVV @VVV @VVV\\
{i_{X}}_* {j_U}_! {j_U}^*  @>>> {i_X}_* @>>> {i_{X}}_* {i_Z}_* {i_Z}^*
@>+1>>
\end{CD}
\end{equation}

Here the middle vertical map is the identity, and the left and right
vertical isomorphisms are induced by adjunction. Plainly, the isomorphism
\eqref{ddt} comes from an isomorphism of short exact sequences of
functors for the abelian categories of sheaves of abelian groups.

Translating our analysis of $\pH^0 {j_U}_! {j_U}^* \cP$ into the
corresponding statement for $\g$-modules and using Equation \eqref{ddt},
we obtain a surjection of highest weight modules:
\[
\Gamma(G/B, H^0 {j_{\tu}}_! {j_{\tu}}^* \cV) \rightarrow V \rightarrow
0,
\]

\noindent where the former module has integrability $J$. To finish the
proof of Theorem \ref{brg}, it remains only to identify the dual of
$\Gamma(G/B, H^0 {j_{\tu}}_! {j_{\tu}}^* \cV)$ with sections of
$\mathbb{D} \cV$ on $\tu$. But using standard compatibilities of
$\mathbb{D}$, and of composition of derived functors, we obtain: 
\begin{align*}
\mathbb{D} \Gamma(G/B, H^0 {j_{\tu}}_! {j_{\tu}}^* \cV) &\simeq
\Gamma(G/B, \mathbb{D} H^0 {j_{\tu}}_! {j_{\tu}}^* \cV)\\
&\simeq \Gamma(G/B, H^0 \mathbb{D} {j_{\tu}}_! j_{\tu}^* \cV)  \\
&\simeq \Gamma(G/B, H^0 R{j_{\tu}}_* {j_{\tu}}^* \mathbb{D} \cV)\\
&\simeq R^0 \Gamma(G/B, R {j_{\tu}}_* j_{\tu}^* \mathbb{D} \cV) \\
&\simeq \Gamma(\tu, \mathbb{D} \cV).\qedhere
\end{align*}
\end{proof}

If one thinks about the above proof, in fact all we used about $U$ (and
$\tilde{U}$) was the intersection of $U$ with $\cU$. The following
proposition shows our $U$ has the correct components in codimension
$\geqslant 2$ to mimic another feature of Brion's example: 

\begin{pro}
The construction of Theorem \ref{brg} sends parabolic Verma modules to
parabolic Verma modules.
\end{pro}

\begin{proof}
For $J \subset K \subset I_{L(w \cdot - 2\rho)},$ we know that $M(w
\cdot - 2 \rho, K)$ corresponds to $j_! \C_{P_KwB/B}$, where $P_K$ is the
parabolic subgroup corresponding to $K \subset I$. Applying the
construction (before taking perverse cohomology), we obtain ${j_{U}}_!
j_{U}^* j_! \C_{P_KwB/B} \simeq j_! \C_{U \cap P_KwB/B}$. The claim
follows from the identity $P_KwB/B \setminus Z = P_JwB/B$, i.e. the
identity
\[
W_Kw \setminus \cup_{k \in K \setminus J} \{y \in W: y \leqslant s_kw\} =
W_Jw.
\]

To see this identity, recall by \cite[Exercise 2.26 and proof of
Proposition 2.4.4]{Bjorner-Brenti} that the assignment $w_k w \mapsto w_k
w_\circ$ is an isomorphism of posets $W_Kw \simeq W_K,$ where $w_\circ$
is the longest element of $W_K$. Multiplying the claimed identity on the
right by $w^{-1}$, it is therefore equivalent to:
\[
W_K \setminus \cup_{k \in K \setminus J} \{y \in W_K: y \geqslant s_k \}
= W_J,
\]

\noindent which is clear. 
\end{proof}

While the results in this section concern $\g$ of finite type, we expect
and would be interested to see that similar results hold for $\g$
symmetrizable.

\section{Highest weight modules over symmetrizable quantum
groups}\label{Squantum}

We now extend many results of the previous sections to highest weight
modules over quantum groups $U_q(\lie{g})$, for $\lie{g}$ a Kac--Moody
algebra. Given a generalized Cartan matrix $A$, as for $\g = \lie{g}(A)$,
to write down a presentation for the algebra $U_q(\lie{g})$ via
generators and explicit relations, one uses the symmetrizability of $A$.
When $\g$ is non-symmetrizable, even the formulation of $U_q(\g)$ is
subtle and is the subject of recent research \cite{Fa}. In light of
this, we restrict to $U_q(\g)$ where $\lie{g}$ is symmetrizable.

\subsection{Notation and preliminaries}

We begin by reminding standard definitions and notation. Fix $\lie{g} =
\lie{g}(A)$ for $A$ a symmetrizable generalized Cartan matrix. Fix a
diagonal matrix $D = {\rm diag}(d_i)_{i \in I}$ such that $DA$ is
symmetric and $d_i \in \mathbb{Z}^{>0},\ \forall i \in I$. Let
$(\h, \pi, \hpi)$ be a realization of $A$ as before; further fix a
lattice $P^\vee \subset \h$, with $\Z$-basis $\halpha_i, \check\beta_l,\
i \in I, \ 1 \leqslant l \leqslant |I| - \rk(A)$, such that $P^\vee
\otimes_\Z \C \simeq \h$ and $(\check\beta_l, \alpha_i) \in \Z,\ \forall
i \in I,\ 1 \leqslant l \leqslant |I|-\rk(A)$.
Set $P := \{ \lambda \in \lie{h}^* : (P^\vee, \lambda) \subset \Z \}$ to
be the weight lattice.
We further retain the notations $\rho, \lie{b}, \lie{h}, \lie{l}_J,
P^+_J, M(\lambda,J), I_\lambda$ from previous sections; note we may
and do choose $\rho \in P$. 
We normalize the Killing form $(\cdot,\cdot)$ on $\lie{h}^*$ to satisfy:
$(\alpha_i, \alpha_j) = d_i a_{ij}$ for all $i,j \in I$.

Let $q$ be an indeterminate. Then the corresponding quantum Kac--Moody
algebra $U_q(\lie{g})$ is a $\C(q)$-algebra, generated by elements $f_i,
q^h, e_i,\ i \in I,\ h \in P^\vee$, with relations given in
e.g.~\cite[Definition 3.1.1]{HK}. Among these generators are
distinguished elements $K_i = q^{d_i \halpha_i} \in q^{P^\vee}$. Also
define $U_q^\pm$ to be the subalgebras generated by the $e_i$ and the
$f_i$, respectively.

A \textit{weight} of the quantum torus $\mathbb{T}_q :=
\C(q)[q^{P^\vee}]$ is a $\C(q)$-algebra homomorphism $\chi : \mathbb{T}_q
\to \C(q)$, which we identify with an element $\mu_q \in
(\C(q)^\times)^{2|I| - \rk(A)}$ given an enumeration of $\halpha_i,\ i
\in I$. We will abuse notation and write $\mu_q(q^h)$ for
$\chi_{\mu_q}(q^h)$. There is a partial ordering on the set of weights,
given by: $q^{-\nu} \mu_q \leqslant \mu_q$, for all weights $\nu \in
\Z^{\geqslant 0} \pi$. We will mostly be concerned with \textit{integral}
weights $\mu_q = q^\mu$ for $\mu \in P$, which are defined via:
$q^\mu(q^h) = q^{(h,\mu)},\ h \in P^\vee$.

Given a $U_q(\lie{g})$-module $V$ and a weight $\mu_q$, the
corresponding \textit{weight space} of $V$ is:
\[
V_{\mu_q} := \{ v \in V : q^h v = \mu_q(q^h) v\ \forall h \in P^\vee \}.
\]

\noindent Denote by $\wt V$ the set of weights $\{ \mu_q : V_{\mu_q} \neq
0 \}$.
A $U_q(\g)$-module is \textit{highest weight} if there exists a
nonzero weight vector which generates $V$ and is killed by $e_i,\ \forall
i \in I$.
For a weight $\mu_q$, let $M(\mu_q), L(\mu_q)$ denote the Verma and
simple $U_q(\g)$-modules of highest weight $\mu_q$, respectively.

Writing $\lambda_q$ for the highest weight of $V$, the
\textit{integrability} of $V$ equals:
\begin{equation}
I_V := \{ i \in I : \dim \C(q)[f_i] V_{\lambda_q} < \infty \}.
\end{equation}

\noindent In this case, the parabolic subgroup $W_{I_V}$ acts on $\wt V$
by
\[
s_i(\lambda_q)(q^h) := \lambda_q(q^{\halpha_i})^{-\alpha_i(h)}
\lambda_q(q^h), \qquad h \in P^\vee.
\]

\noindent The braid relations can be checked using by specializing $q$ to
$1$, cf.\ \cite{Lus88}. In particular, $w(q^\lambda) = q^{w \lambda}$ for
$w \in W$ and $\lambda \in P$.

Given a weight $\lambda_q$ and $J \subset I_{L(\lambda_q)}$, the
parabolic Verma module $M(\lambda_q, J)$ co-represents the functor:
\begin{equation}
M \rightsquigarrow \{ m \in M_{\lambda_q} : e_i m = 0\ \forall i \in I,
\quad f_j \text{ acts nilpotently on } m, \forall j \in J \}.
\end{equation}

Note $\lambda_q(q^{\halpha_j}) = \pm q^{n_j}$, $n_j \geqslant 0,\ \forall
j \in J$ cf.\ \cite[Proposition 2.3]{Jantzen-book}. Therefore:
\begin{equation}\label{Eqgvm}
M(\lambda_q,J) \simeq M(\lambda_q)/( f_j^{n_j + 1}
M(\lambda_q)_{\lambda_q}, \forall j \in J).
\end{equation}

In what follows, we will specialize highest weight modules at $q=1$, as
pioneered by Lusztig \cite{Lus88}. Let $A_1$ denote the local ring of
rational functions $f \in \C(q)$ that are regular at $q=1$. Fix an
integral weight $\lambda_q = q^\lambda \in q^P$ and a highest weight
module $V$ with highest weight vector $v_{\lambda_q} \in V_{\lambda_q}$.
Then the classical limit of $V$ at $q=1$, defined to be $V^1 := A_1 \cdot
v_{\lambda_q} / (q-1) A_1 \cdot v_{\lambda_q}$, is a highest weight
module over $U (\g)$ with highest weight $\lambda$. Moreover the
characters of $V$ and $V^1$ are ``equal'', i.e.~upon identifying $q^P$
with $P$.

\subsection{Three invariants of highest weight modules}

We begin by extending Theorem \ref{maintheo} to quantum groups.

\begin{theo}\label{Tqequiv}
Let $V$ be a highest weight module with integral highest weight
$\lambda_q$. The following data are equivalent:
\begin{enumerate}
\item $I_V$, the integrability of $V$.
\item $\conv V^1$, the convex hull of the specialization of $V$.
\item The stabilizer of $\ch V$ in $W$.
\end{enumerate}
\end{theo}

\begin{proof}
The equivalence of the three statements follows from their classical
counterparts, using the equality of characters under specialization.
\end{proof}

\subsection{Three positive formulas for the weights of simple modules}

We now extend the results of Section \ref{Sbump} to quantum groups.

\begin{theo}\label{Tqwts}
Let $\lambda_q$ be integral.
The following are equivalent:
\begin{enumerate}
\item $\wt V = \wt M(\lambda_q, I_V)$.
\item $\wt_{I^p_V} V = q^{-\Z^{\geqslant 0} \pi_{I^p_V}} \lambda_q$,
where $I^p_V = I_{L(\lambda_q)} \setminus I_V$.
\end{enumerate} 

\noindent In particular, if $V$ is simple, or more generally $|I^p_V|
\leqslant 1$, then $\wt V = \wt M(\lambda_q, I_V)$.
\end{theo}

To prove Theorem \ref{Tqwts}, we will need the Integrable Slice
Decomposition for quantum parabolic Verma modules.

\begin{pro}\label{Pqint}
Let $\lambda_q$ be an integral weight and $M(\lambda_q,J) =: V$ a
parabolic Verma module. Then:
\begin{equation}\label{Eqint}
\wt M(\lambda_q, J) = \bigsqcup_{\mu \in \Z^{\geqslant 0} (\pi \setminus
\pi_J)} \wt L_{\hspace{.5mm} \fl_J}(q^{-\mu} \lambda_q),
\end{equation}

\noindent where $L_{\hspace{.5mm} \fl_J}(\nu)$ denotes the simple
$U_q(\fl_J)$-module of highest weight $\nu$. In particular, $\wt V^1 =
\wt M(\lambda_1,J)$.
\end{pro}

\begin{proof}
It suffices to check the formula after specialization. Since the
characters of $V$ and $V^1$ are ``equal'', $I_V = I_{V^1}$. It therefore
suffices to check the surjection $M(\lambda_1,J) \to V^1$ induces an
equality of weights.
By the Integrable Slice Decomposition \ref{slice}, it suffices to show
that $q^{-\Z^{\geqslant 0} \pi_{I \setminus J}} \lambda_q \subset \wt
M(\lambda_q,J)$.
But this is clear from \eqref{Eqgvm} by considering weights and using
that $f_j^{n_j+1} M(\lambda_q)_{\lambda_q}$ is a highest weight line, for
all $j \in J$.
\end{proof}

\begin{re}\label{Rqpbw}
If $\g$ is of finite type, then in fact $V^1 \simeq M(\lambda_1,J)$. The
proof uses \cite[Chapter 2]{BNPP} and the PBW theorem, as well as
arguments similar to the construction of Lusztig's canonical basis
\cite{Lus90} to define $U_q(\lie{u}_J^-)$.
It would be interesting to know if the parabolic Verma module
$M(\lambda_q,J)$ specializes to $M(\lambda_1,J)$ for all symmetrizable
$\g$.
\end{re}

\begin{proof}[Proof of Theorem \ref{Tqwts}]
By Proposition \ref{Pqint}, (1) implies (2).
For the converse, by the Integrable Slice Decomposition \eqref{Eqint}, it
suffices to show that $q^{-\Z^{\geqslant 0} (\pi \setminus \pi_{I_V})}
\lambda_q \subset \wt V$. This follows by mimicking the proof of Theorem
\ref{weights}.
\end{proof}

As an application of Theorem \ref{Tqwts}, we obtain the following
positive formulas for weights of simple modules $\wt L(\lambda_q)$.

\begin{pro}
Suppose $\lambda_q$ is integral.
Write $\fl$ for the Levi subalgebra corresponding to $I_{L(\lambda_q)}$,
and write $L_\fl(\nu_q)$ for the simple $U_q(\fl)$ module with highest
weight $\nu_q$. Then:
\begin{equation}\label{Epositive1}
\wt L(\lambda_q) = \bigsqcup_{\mu \in \Z^{\geqslant 0} \pi \setminus
\pi_{I_{L(\lambda_q)}} } \wt L_{\mathfrak{l}}(q^{-\mu} \lambda_q).
\end{equation}
\end{pro}

\begin{pro}\label{Pqbump}
Let $\lambda_q$ be integral and $V := L(\lambda_q)$. Then:
\begin{equation}\label{Eqsimple}
\wt V^1 = (\lambda + \Z \pi) \cap \conv V^1.
\end{equation}

\noindent By specialization, this determines the weights of
$L(\lambda_q)$.
\end{pro}
\begin{pro}
Suppose $\lambda_q$ is integral and has finite
$W_{I_{L(\lambda_q)}}$-isotropy. Then:
\begin{equation}\label{Epositive2}
\wt L(\lambda_q) = \bigcup_{w \in W_{I_{L(\lambda_q)} }} w \{ q^\nu : \nu
\in P^+_{I_{L(\lambda_q)}},\ q^\nu \leqslant \lambda_q \}.
\end{equation}
\end{pro}

\subsection{Weights and convex hulls of highest weight modules}

The result and arguments of Section \ref{Slepowsky} apply without change
to quantum groups.

\begin{theo}\label{TqLepowsky}
Fix $J \subset I_{L(\lambda_q)}$ and define $J_q^p := I_{L(\lambda_q)}
\setminus J$. The following are equivalent:
\begin{enumerate}
\item $\wt V = \wt M(\lambda_q,J)$ for every $V$ with $I_V = J$.

\item The Dynkin diagram for $\lie{g}_{J_q^p}$ is complete.
\end{enumerate}
\end{theo}

\subsection{An alternating formula for the weights of simple modules}

The main result of Section \ref{Sweylkac} follows from combining Theorems
\ref{Tsimple}, \ref{Tqwts}, Proposition \ref{Pqint}, and specializing.

\begin{theo}\label{qwwks}
For all integral $\lambda_q$ with finite stabilizer in
$W_{I_{L(\lambda_q)}}$, we have:
\begin{equation}
\wt L(\lambda_q) = \sum_{w \in W_{I_{L(\lambda_q)}} } w
\frac{\lambda_q}{\prod_{\alpha \in \pi} (1 - q^{-\alpha})}.
\end{equation}
\end{theo}

\subsection{The case of non-integral weights}\label{Snonint}

We now explain how to extend the above results in this section to other
highest weights. We do so in two ways. First, we observe that with some
modifications, Lusztig's specialization method applies to any weight
$\lambda_q$ such that $\lambda_q(q^h)$ is regular at $q=1$ with value
$1,\ \forall h \in P^\vee$. Calling such weights \textit{specializable},
one can show that as before, a highest weight module with specializable
highest weight $\lambda_q$ specializes to a highest weight $U(\g)$-module
with highest weight $\lambda_1 \in \h^*$, given by:
\[
(h,\lambda_1) := \left. \frac{\lambda_q(q^h)-1}{q-1} \right|_{q=1},
\qquad h \in P^\vee.
\]

\noindent Moreover, the following holds:

\begin{theo}
The above results in this section all extend to $\lambda_q$
specializable.
\end{theo}

Second, we extend many of the above results to generic highest weights,
i.e.~with finite integrable stabilizer. In particular, this covers all
cases in finite and affine type, the remaining cases in affine type being
trivial modules. As before, we obtain:

\begin{theo}\label{Tqwts2}
Fix a weight $\lambda_q$ and a highest weight module $V$ such that the
stabilizer of its highest weight $\lambda_q$ in $W_{I_V}$ is finite. The
following are equivalent:
\begin{enumerate}
\item $\wt V = \wt M(\lambda_q, I_V)$.
\item $\wt_{I^p_V} V = q^{-\Z^{\geqslant 0} \pi_{I^p_V}} \lambda_q$,
where $I^p_V = I_{L(\lambda_q)} \setminus I_V$.
\end{enumerate} 

\noindent In particular, if $V$ is simple, or more generally $|I^p_V|
\leqslant 1$, then $\wt V = \wt M(\lambda_q, I_V)$.
\end{theo}

To prove Theorem \ref{Tqwts2}, we will need the Integrable Slice
Decomposition for quantum parabolic Verma modules.

\begin{pro}\label{Pqint2}
Let $M(\lambda_q,J)$ be a parabolic Verma module such that the
stabilizer of $\lambda_q$ in $W_J$ is finite. Then:
\begin{equation}\label{Eqint2}
\wt M(\lambda_q, J) = \bigsqcup_{\mu \in \Z^{\geqslant 0} (\pi \setminus
\pi_J)} \wt L_{\hspace{.5mm} \fl_J}(q^{-\mu} \lambda_q),
\end{equation}

\noindent where $L_{\hspace{.5mm} \fl_J}(\nu)$ denotes the simple
$U_q(\fl_J)$-module of highest weight $\nu$.
\end{pro}

\begin{proof}
The inclusion $\supset$ follows from considering weights as in
Proposition \ref{slice}. The inclusion $\subset$ follows by using
Proposition \ref{nonde}(2), which holds for quantum groups by
specialization to $q = 1$.
\end{proof}

\begin{proof}[Proof of Theorem \ref{Tqwts2}]
By Proposition \ref{Pqint2}, (1) implies (2).
For the converse, by the Integrable Slice Decomposition \eqref{Eqint2},
it suffices to show that $q^{-\Z^{\geqslant 0} (\pi \setminus \pi_{I_V})}
\lambda_q \subset \wt V$. This follows by mimicking the proof of Theorem
\ref{weights}.
\end{proof}

As an application of Theorem \ref{Tqwts2}, Equation \eqref{Epositive1}
holds on the nose for all $\lambda_q$ with finite stabilizer in
$W_{I_{L(\lambda_q)}}$, and Equation \eqref{Epositive2} holds, rephrased
as follows:

\begin{pro}
Suppose $\lambda_q$ has finite $W_{I_{L(\lambda_q)}}$-isotropy. Then:
\begin{equation}
\wt L(\lambda_q) = \bigcup_{w \in W_{I_{L(\lambda_q)} }} w \{ \nu_q :
\nu_q \leqslant \lambda_q, \ \nu_q(q^{\halpha_i}) = \pm q^{n_i}, n_i
\in \Z^{\geqslant 0},\ \forall i \in I_{L(\lambda_q)} \}.
\end{equation}
\end{pro}

Finally, we extend the results of Section \ref{Slepowsky} to quantum
groups, as the same proof applies.

\begin{theo}
Fix $\lambda_q$ and $J \subset I_{L(\lambda_q)}$ such that the stabilizer
of $\lambda_q$ in $W_J$ is finite, and define $J_q^p := I_{L(\lambda_q)
\setminus J}$. The following are equivalent:
\begin{enumerate}
\item $\wt V = \wt M(\lambda_q,J)$ for every $V$ with $I_V = J$.

\item The Dynkin diagram for $\lie{g}_{J_q^p}$ is complete.
\end{enumerate}
\end{theo}

\def\cprime{$'$}


\end{document}